\documentclass[a4paper]{amsart}
\usepackage[utf8]{inputenc}
\usepackage{enumerate}
\usepackage{fullpage}
\usepackage{amsaddr}
\usepackage{mathtools}
\usepackage{amsthm}
\usepackage{amsmath}
\usepackage{mathrsfs}
\usepackage{array}
\usepackage{bbm}
\usepackage{fourier}
	\usepackage{faktor}
\usepackage{amssymb}
\usepackage{amsfonts}
\usepackage{abstract}
\usepackage{tikz}
\usetikzlibrary{arrows,snakes}
\usepackage[all,cmtip,2cell]{xy}
\usepackage[left=1in + \hoffset, right=1in +\hoffset
]{geometry}

\usepackage[backend=bibtex,style=alphabetic,sorting=nty,doi=false,url=false,eprint=false]{biblatex}
\usepackage{fancyhdr}
\newcommand{\blank}{\underline{\hspace{0.2cm}}}
\def\bign#1{\mathclose{\hbox{$\left#1\vbox to8.5\p@{}\right.\n@space$}}\mathopen{}}
\swapnumbers
\theoremstyle{definition}
\newtheorem{theorem}{Theorem}[]
\newtheorem*{theorem*}{Theorem}
\newtheorem{lemma}[theorem]{Lemma}
\newtheorem*{lemma*}{Lemma}
\newtheorem{corollary}[theorem]{Corollary}
\newtheorem*{corollary*}{Corollary}
\newtheorem{definition}[theorem]{Definition}
\newtheorem*{definition*}{Definition}

\newtheorem*{remark*}{Remark}
\newtheorem{proposition}[theorem]{Proposition}
\newtheorem*{proposition*}{Proposition}
\newtheorem{example}[theorem]{Example}
\newtheorem*{example*}{Example}
\usepackage[backend=bibtex,style=alphabetic,sorting=nty,doi=false,url=false,eprint=false]{biblatex}
\makeatletter
\renewcommand{\email}[2][]{%
	\ifx\emails\@empty\relax\else{\g@addto@macro\emails{,\space}}\fi%
	\@ifnotempty{#1}{\g@addto@macro\emails{\textrm{(#1)}\space}}%
	\g@addto@macro\emails{#2}%
}
\makeatother
\title{Exact Borel subalgebras of path algebras of quivers of Dynkin type $\mathbb{A}$}
\author{Markus Thuresson}
\address{Department of Mathematics, Uppsala University, Box 480, SE-75106, Sweden}
\email{markus.thuresson@math.uu.se}
\begin{document}
\maketitle

\begin{abstract}
Hereditary algebras are quasi-hereditary with respect to any adapted partial order on the indexing set of the isomorphism classes of their simple modules. For any adapted partial order on $\{1,\dots, n\}$, we compute the quiver and relations for the $\operatorname{Ext}$-algebra of standard modules over the path algebra of a uniformly oriented linear quiver with $n$ vertices. Such a path algebra always admits a regular exact Borel subalgebra in the sense of König and we show that there is always a regular exact Borel subalgebra containg the idempotents $e_1,\dots, e_n$ and find a minimal generating set for it. For a quiver $Q$ and a deconcatenation $Q=Q^1\sqcup Q^2$ of $Q$ at a sink or source $v$, we describe the $\operatorname{Ext}$-algebra of standard modules over $KQ$, up to an isomorphism of associative algebras, in terms of that over $KQ^1$ and $KQ^2$. Moreover, we determine necessary and sufficient conditions for $KQ$ to admit a regular exact Borel subalgebra, provided that $KQ^1$ and $KQ^2$ do. We use these results to obtain sufficient and necessary conditions for a path algebra of a linear quiver with arbitrary orientation to admit a regular exact Borel subalgebra.
\end{abstract}

\thispagestyle{empty}
\newpage
\noindent
\section{Introduction}
Highest weight categories were first introduced in \cite{CPS88} with the purpose of axiomatizing certain phenomena arising naturally in the representation theory of complex semisimple Lie algebras. This prompted the definition of a quasi-hereditary algebra (see \cite{Scott87}), which was shown by the authors to be precisely such a finite-dimensional algebra whose module category is equivalent to a highest weight category. Since their introduction, examples of quasi-hereditary algebras have been found to be abundant in representation theory. Classical examples include hereditary algebras, algebras of global dimension two, Schur algebras and blocks of BGG category $\mathcal{O}$. The main protagonists of the representation theory of quasi-hereditary algebras are the standard modules. These are certain quotients of the indecomposable projective modules which depend on a chosen partial order on the indexing set of the isomorphism classes of simple modules. For blocks of BGG category $\mathcal{O}$, this order is the Bruhat order. Associated to the standard modules is the category $\mathcal{F}(\Delta)$: the full subcategory of the module category consisting of those modules which admit a filtration whose subquotients are standard modules. Given a quasi-hereditary algebra, natural objectives of its representation theory are understanding the standard modules and the associated category $\mathcal{F}(\Delta)$.

In the endeavor of understanding $\mathcal{F}(\Delta)$ for a general quasi-hereditary algebra, a breakthrough was made by König, Külshammer and Ovsienko in \cite{kko}. They showed that any quasi-hereditary algebra $A$ is Morita equivalent to an algebra $\Lambda$, which admits a very particular subalgebra $B$, called a regular exact Borel subalgebra. This subalgebra has the surprising property that the image of its module category under the functor $\Lambda\otimes_B \blank$ is equivalent to $\mathcal{F}(\Delta)$. 

In general, the problem of determining $\Lambda$ and $B$ for an arbitrary quasi-hereditary algebra $A$ may be very hard. One way to do this involves determining an $A_\infty$-structure on $\operatorname{Ext}_A^\ast(\Delta,\Delta)$, the algebra of extensions between standard modules over $A$, which can be arduous. Some examples of where this has been done can be found in \cite{KlamtStroppel, Thuresson22}. However, in the situations studied in this paper, the aforementioned $A_\infty$-structures can be easily computed. Moreover, we are able to use criteria from \cite{CONDE21} to check when our algebras themselves contain regular exact Borel subalgebras.

A given associative algebra can have different quasi-hereditary structures, depending on the partial order mentioned above. In particular, two different orders may yield identical quasi-hereditary structures. This phenomenon motivates the definition of the esssential order, an order with the property that two quasi-hereditary structures coincide if and only if they generate the same essential order. The algebras considered in this article are hereditary, and hereditary algebras are precisely those algebras which are quasi-hereditary with respect to any total order on an indexing set of the isomorphism classes of their simple modules. A natural question to ask is then how many different quasi-hereditary structures there are. When the algebra in question admits a duality on its module category which preserves simple modules, the structure is essentially unique, as shown by Coulembier in \cite{Kevin2020}. This contrasts the algebras studied in this article with the examples from \cite{KlamtStroppel, Thuresson22}, as the algebras appearing there admit a simple-preserving duality.

In a recent paper by Flores, Kimura and Rognerud, \cite{FKR}, the authors showed that when $A$ is the path algebra of a uniformly oriented linear quiver, the different quasi-hereditary structures are counted by the Catalan numbers. Moreover, they counted the quasi-hereditary structures of path algebras of more complicated quivers by means of a combinatorial process called ``deconcatenation''. The main goal of the present article is to use the combinatorial techniques of Flores, Kimura and Rognerud to further study the quasi-hereditary structures of some of the algebras considered in \cite{FKR}. In particular, we want to describe the $\operatorname{Ext}$-algebras of their standard modules and their regular exact Borel subalgebras.
 
 The following is a brief description of the main results of the present article. For a quasi-hereditary algebra $A$, denote by $\Delta$ the direct sum of standard modules, one from each isomorphism class, and denote by $\operatorname{Ext}_{A}^\ast(\Delta,\Delta)$ the algebra of extensions between standard modules.
 
 \begin{enumerate}[(A)]
 	\item Let $A_n$ be the path algebra of the following quiver:
 	$$\xymatrix{1\ar[r] & 2 \ar[r] & \dots \ar[r] & n-1\ar[r] & n}$$
 	For any partial order $\trianglelefteq$ on $\{1,\dots, n\}$ such that $(A_n,\trianglelefteq)$ is quasi-hereditary, we construct a graded quiver $Q$ and an admissible ideal $I\subset KQ$, such that there is an isomorphism of graded associative algebras $\operatorname{Ext}_{A_n}^\ast(\Delta,\Delta)\cong KQ/I$. According to results from \cite{FKR}, each quasi-hereditary structure corresponds to a unique binary tree, and we show that this tree encodes all necessary information about extensions between the standard modules. For all details, we refer to Section~3.
 	\item Let $A_n$ be as in (A). For any partial order $\trianglelefteq$ on $\{1,\dots, n\}$ such that $(A_n,\trianglelefteq)$ is quasi-hereditary, we check that $A_n$ admits a regular exact Borel subalgebra. We show that $A_n$ always admits a regular exact Borel subalgebra $B$ which contains the idempotents $e_1,\dots, e_n$ and find a minimal generating set for $B$. For all details, we refer to Section~3.2.
 	\item Let $Q=Q^1\sqcup Q^2$ be a deconcatenation of the quiver $Q$ at a sink or source $v$. Put $A=KQ$ and $A^\ell=KQ^\ell$, for $\ell=1,2$. We describe $\operatorname{Ext}_A^\ast(\Delta,\Delta)$ up to isomorphism in terms of the Ext-algebras of standard modules over $A^\ell$, via a certain ``gluing'' process.
 	\item Given that $A^\ell$ admits a regular exact Borel subalgebra $B^\ell$, for $\ell=1,2$, we show the following:
 	\begin{enumerate}[(i)]
 		\item If $v$ is a source, $A$ admits a regular exact Borel subalgebra.
 	\item If $v$ is a sink, then $A$ admits a regular exact Borel subalgebra if and only if $v$ is minimal or maximal with respect to the essential order on $Q_0$.
 	\end{enumerate}
 In the cases where $A$ admits a regular exact Borel subalgebra, we construct a regular exact Borel subalgebra of $A$ from $B^1$ and $B^2$, using a similar ``gluing'' as in (C).
 \end{enumerate}
This article is organized in the following way: In Section~2 we give the necessary background on quasi-hereditary algebras and fix some notation. In Section~3, we recall the results of \cite{FKR} on the quasi-hereditary structures of $A_n$ and expand upon them, proving (A). In Subsection~3.1, we compute the quiver and relations of the Ringel dual of $A_n$. Subsection~3.2 is devoted to the description of the regular exact Borel subalgebras of $A_n$. Subsection~3.3 briefly discusses $A_\infty$-structures on $\operatorname{Ext}_{A_n}^\ast(\Delta,\Delta)$.

In Section~4, we give the background on deconcatenations from \cite{FKR} and prove (C). Subsection~4.1 discusses how regular exact Borel subalgebras behave under deconcatenations and contains the proof of (D). In Subsection~4.2, we apply our results to the case where $Q$ is a linear quiver with arbitrary orientation.
\newpage
\tableofcontents
\newpage
\section{Notation and background}
Throughout, let $K$ be an algebraically closed field. For a quiver $Q=(Q_0, Q_1)$, denote by $KQ$ path algebra of $Q$. Let $I\subset KQ$ be an admissible ideal and let $A$ be the quotient $A=KQ/I$. We take $A$-module to mean finite-dimensional left $A$-module if nothing else is stated. For an arrow $\alpha\in Q_1$, denote by $s(\alpha)$ and $t(\alpha)$ the starting and terminal vertex of $\alpha$, respectively. We adopt the convention of writing composition of arrows from right to left. That is, for vertices and arrows arranged as
$$\xymatrix{x\ar[r]^-{\alpha} & y\ar[r]^-{\beta} &z,}$$
we write the composition  ``first $\alpha$, then $\beta$'' as $\beta \alpha$. We extend the notation of starting and terminal vertex to paths in $Q$, so if $p=\alpha_n \dots \alpha_1$, then we put $s(p)=s(\alpha_1)$ and $t(p)=t(\alpha_n)$.

The isomorphism classes of the simple $A$-modules are indexed by the vertex set, $Q_0$, of $Q$. Denote by $L(i)$ the simple $A$-module associated to the vertex $i$. Denote by $P(i)$ and $I(i)$ the projective cover and injective envelope of $L(i)$, respectively. For an $A$-module $M$ and a simple module $L(i)$, denote by $[M:L(i)]$ the Jordan-Hölder multiplicity of $L(i)$ in $M$. We denote the category of finite-dimensional left $A$-modules by $A\operatorname{-mod}$.
\begin{definition}\cite{CPS88}
	Let $A$ be a finite-dimensional algebra. Let $\{1,\dots, n\}$ be an indexing set for the isomorphism classes of simple $A$-modules and let $\trianglelefteq$ be a partial order on $\{1,\dots, n\}$. Let the \emph{standard module} at $i$, denoted by $\Delta(i)$, be the largest quotient of $P(i)$, whose composition factors $L(j)$ are such that $j\trianglelefteq i$. The algebra $A$ is said to be \emph{quasi-hereditary} with respect to $
	\trianglelefteq$ if the following hold.
	\begin{enumerate}[(QH1)]
		\item  There is a surjection $P(i)\twoheadrightarrow \Delta(i)$ whose kernel admits a filtration with subquotients $\Delta(j)$, where $j\triangleright i$.
		\item There is a surjection $\Delta(i) \twoheadrightarrow L(i)$ whose kernel admits a filtration with subquotients $L(j)$, where $j\triangleleft i$.
	\end{enumerate}
	Equivalently, let the \emph{costandard module} at $i$, denoted by $\nabla(i)$, be the largest submodule of $I(i)$, whose composition factors are such that $j\trianglelefteq i$. The algebra $A$ is said to be \emph{quasi-hereditary} with respect to $\trianglelefteq$ if the following hold.
	\begin{enumerate}[(QH1)$^\prime$]
		\item There is an injection $\nabla(i)\hookrightarrow I(i)$ whose cokernel admits a filtration with subquotients $\nabla(j)$, where $j\triangleright i$.
		\item There is an injection $L(i)\hookrightarrow \nabla(i)$ whose cokernel admits a filtration with subquotients $L(j)$, where $j\triangleleft i$.
	\end{enumerate}
\end{definition}
For a quasi-hereditary algebra $A$, it is natural to consider two particular subcategories of its module category, $A\operatorname{-mod}$. These are $\mathcal{F}(\Delta)$, the full subcategory of $A\operatorname{-mod}$ consisting of those modules which admit a filtration by standard modules, and $\mathcal{F}(\nabla)$, the full subcategory of $A\operatorname{-mod}$ consisting of those modules which admit a filtration by costandard modules.

For an $A$-module $M\in \mathcal{F}(\Delta)$ (or $\mathcal{F}(\nabla)$), the number of occurrences of a particular standard module $\Delta(i)$ (or costandard module $\nabla(i)$) as a subquotient of a filtration of $M$ is well-defined, and we denote this number by $(M:\Delta(i))$ (or $(M:\nabla(i))$).

In general, refining the partial order $\trianglelefteq$ may produce different standard and costandard modules. Traditionally, only so-called adapted orders on $\{1,\dots, n\}$ are considered in order to avoid this.
\begin{definition}\cite{DlabRingel}
 Let $A$ be a finite-dimensional algebra. Let $\{1,\dots, n\}$ be an indexing set for the isomorphism classes of simple $A$-modules and let $\trianglelefteq$ be a partial order on $\{1,\dots, n\}$. We say that $\trianglelefteq$ is \emph{adapted} to $A$ if and only if for any $A$-module $M$ such that
$$\operatorname{top}M\cong L(i)\quad \textrm{and}\quad \operatorname{soc}M\cong L(j),$$
where $i$ and $j$ are incomparable, there is $1\leq k\leq n$ such that $i\triangleleft k$, $j\triangleleft k$ and $[M:L(k)]>0$.
\end{definition}
\begin{lemma}\cite{CondeThesis} \label{lemma: q.h algebra has adapted order}
	If $(A, \trianglelefteq)$ is a quasi-hereditary algebra, then $\trianglelefteq$ is adapted to $A$.
\end{lemma}
	Two quasi-hereditary structures $(A, \trianglelefteq_1)$ and $(A,\trianglelefteq_2)$ are said to be equivalent if the sets of standard (and costandard) modules with respect to $\trianglelefteq_1$ and $\trianglelefteq_2$ coincide. We denote this relationship by $\trianglelefteq_1 \sim \trianglelefteq_2$. Then, more precisely:
$$\trianglelefteq_1\sim\trianglelefteq_2 \iff \Delta_1(i)=\Delta_2(i)\wedge \nabla_1(i)=\nabla_2(i),\ \forall i\in Q_0.$$
Equivalence of different quasi-hereditary structures is also captured precisely by the \emph{essential order}.
\begin{definition}\cite[Definition~1.2.5]{Kevin2020}
	Let $(A,\trianglelefteq)$ be a quasi-hereditary algebra. Define the \emph{essential order} $\trianglelefteq^e$ of $\trianglelefteq$ as the partial order transitively generated by the relations
	$$i\trianglelefteq^e j \iff \left[\Delta(j):L(i)\right] >0 \quad \textrm{or}\quad (P(i):\Delta(j))>0.$$
\end{definition}
The essential order is related to equivalence of quasi-hereditary structures via
$$\trianglelefteq_1\sim \trianglelefteq_2 \iff \trianglelefteq_1^e=\triangleleft_2^e.$$
\subsection{Gluing of subalgebras}
At various points in the present article, there will appear algebras which, intuitively, arise from gluing two subspaces of some ambient algebra at a shared dimension. Let $X$ and $Y$ be subpaces of some algebra $A$, which are closed under multiplication. Assume that $\dim X\cap Y=1$ and choose complements $X^\prime$ and $Y^\prime$ of $X\cap Y$ in $X$ and $Y$, respectively. Assume moreover that $X^\prime \cdot Y^\prime=Y^\prime \cdot X^\prime =0$ and that $1_A\in \operatorname{span}(X, Y)$. Put
$$C=X^\prime \oplus Y^\prime \oplus X\cap Y.$$

Then, $C$ is a subalgebra of $A$. We call $C$ the \emph{gluing} of $X$ and $Y$ and write $C=X\diamond Y$. Note that when $X$ and $Y$ are graded, there is a natural grading on $X\diamond Y$ induced by the gradings on $X$ and $Y$.
\section{The path algebra of $\mathbb{A}_n$}
In this section, we consider the path algebra of the uniformly oriented linear quiver
$$\mathbb{A}_n:\xymatrix{1\ar[r]^-{\alpha_1} & 2 \ar[r]^-{\alpha_2} & \dots \ar[r]^-{\alpha_{n-2}}&n-1 \ar[r]^-{\alpha_{n-1}}& n}.$$
Throughout the section, we put $A_n=K\mathbb{A}_n$. The simple modules over $A_n$ are indexed by the vertices $1,\dots, n$. Moreover, the algebra $A_n$ is hereditary, and therefore, it is quasi-hereditary with respect to any adapted order $(\{1,\dots, n\},\trianglelefteq)$ to $A_n$ \cite{DlabRingel}.
Recall that the indecomposable $A_n$-modules, up to isomorphism, are given by the interval modules, which are modules having Loewy diagrams of the following form:
$$M(i,j):\vcenter{\xymatrixrowsep{0.3cm}\xymatrix{
	i \ar[d] \\ i+1 \ar[d] \\ \vdots \ar[d] \\ j-1 \ar[d]\\ j	
	}}$$
With this notation, we have
$$L(i)=M(i,i),\quad P(i)=M(i,n)\quad \textrm{and}\quad I(i)=M(1,i),\quad \forall 1\leq i\leq n.$$
Homomorphisms and extensions between interval modules are well understood, and we summarize this information in the following proposition. Note that \cite{OpThom10} uses different conventions than the present article.
\begin{proposition}\cite[Theorem~3.6]{OpThom10} \label{proposition:homs and ext between interval modules}
	Let $M(i_1,j_1)$ and $M(i_2, j_2)$ be interval modules. Then, we have
	\begin{enumerate}[(i)]
		\item $$\dim \operatorname{Hom}_{A_n}(M(i_1,j_1), M(i_2, j_2))=\begin{cases}
			1 & \textrm{if } i_2\leq i_1 \leq j_2\leq j_1;\\
			0 & \textrm{otherwise}.
		\end{cases}$$
	\item $$\dim\operatorname{Ext}_{A_n}^1(M(i_1, j_1),M(i_2, j_2))=\begin{cases} 1 & \textrm{if }i_1+1\leq i_2\leq j_1+1\leq j_2;
		\\
		0 & \textrm{otherwise}.
	\end{cases}$$
	\end{enumerate}
\end{proposition}
\begin{definition}
	A \emph{binary tree} $T$ is either the empty set or a triple $(s, L, R)$, where $s$ is a singleton set, called the \emph{root} of $T$, and $L$ and $R$ are two binary trees, called the \emph{left} and \emph{right} subtrees of $s$, respectively. The empty set has one \emph{leaf}, and the set of \emph{leaves} of $T=(s, L, R)$ is the (disjoint) union of the sets of leaves of $L$ and $R$.
	
	A \emph{binary search tree} is a binary tree, whose vertices are labeled by integers, such that if a vertex $x$ is labeled by $k$, then the vertices of the left subtree of $x$ are labeled by integers less than $k$, and the vertices of the right subtree of $x$ are labeled by integers greater than $k$.
	
	If $T$ is a binary tree with $n$ vertices, there exists a unique labelling of the vertices of $T$ by the integers $1,\dots, n$, turning $T$ into a binary search tree. The procedure by which this labelling is obtained is known as the \emph{in-order algorithm}, which recursively visits the left subtree, then the root, then the right subtree. The first vertex visited is labeled by 1, the second by 2, and so on.
\end{definition}
\begin{example}Consider the following binary tree with 6 vertices.
$$	
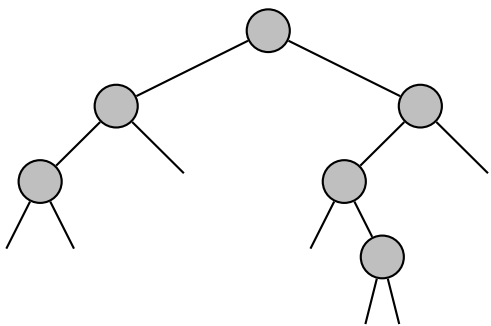
\begin{tikzpicture}
	\node(a) at (0,0) [shape=circle,draw, fill=lightgray, thick] {\phantom{1}};
	\node(b) at (-2,-1) [shape=circle,draw, fill=lightgray, thick] {\phantom{1}};
	\node(c) at (2,-1) [shape=circle,draw, fill=lightgray, thick] {\phantom{1}};
	\node(d) at (-3,-2) [shape=circle,draw, fill=lightgray, thick]  {\phantom{1}};
	\node(e) at (-1,-2)  {};
	\node(f) at (-3.5,-3) {};
	\node(g) at (-2.5,-3) {};
	\node(h) at (3,-2)  {};
	\node(i) at (1,-2) [shape=circle,draw, fill=lightgray, thick] {\phantom{1}};
	\node(j) at (1.5,-3) [shape=circle,draw, fill=lightgray, thick] {\phantom{1}};
	\node(k) at (0.5,-3) {};
	\node(l) at (1.25,-4) {};
	\node(m) at (1.75, -4) {};
	\draw[thick] (a) to (b);
	\draw[thick] (a) to (c);
	\draw[thick] (b) to (d);
	\draw[thick] (b) to (e);
	\draw[thick] (d) to (f);
	\draw[thick] (d) to (g);
	\draw[thick] (c) to (i);
	\draw[thick] (c) to (h);
	\draw[thick] (i) to (j);
	\draw[thick] (i) to (k);
	\draw[thick] (j) to (l);
	\draw[thick] (j) to (m);
	\end{tikzpicture}$$
With the in-order algorithm, the vertices of the tree are labeled as follows, creating a binary search tree.
$$	
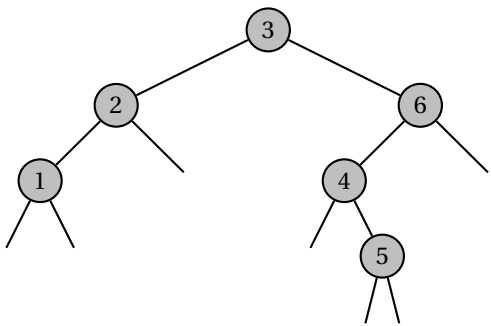
\begin{tikzpicture}
	\node(a) at (0,0) [shape=circle,draw, fill=lightgray, thick] {3};
	\node(b) at (-2,-1) [shape=circle,draw, fill=lightgray, thick] {2};
	\node(c) at (2,-1) [shape=circle,draw, fill=lightgray, thick] {6};
	\node(d) at (-3,-2) [shape=circle,draw, fill=lightgray, thick]  {1};
	\node(e) at (-1,-2)  {};
	\node(f) at (-3.5,-3) {};
	\node(g) at (-2.5,-3) {};
	\node(h) at (3,-2)  {};
	\node(i) at (1,-2) [shape=circle,draw, fill=lightgray, thick] {4};
	\node(j) at (1.5,-3) [shape=circle,draw, fill=lightgray, thick] {5};
	\node(k) at (0.5,-3) {};
	\node(l) at (1.25,-4) {};
	\node(m) at (1.75, -4) {};
	\draw[thick] (a) to (b);
	\draw[thick] (a) to (c);
	\draw[thick] (b) to (d);
	\draw[thick] (b) to (e);
	\draw[thick] (d) to (f);
	\draw[thick] (d) to (g);
	\draw[thick] (c) to (i);
	\draw[thick] (c) to (h);
	\draw[thick] (i) to (j);
	\draw[thick] (i) to (k);
	\draw[thick] (j) to (l);
	\draw[thick] (j) to (m);
\end{tikzpicture}$$
\end{example}
Any binary search tree $T$ with $n$ vertices induces a partial order on the set of its vertices $\{1,\dots, n\}$. We denote this partial order by $\trianglelefteq_T$. It is defined in the following way.
$$i\triangleleft_T j \iff i\textrm{ labels a vertex in the subtree of the vertex labeled by }j.$$ In the above example, the partial order $\trianglelefteq_T$ would be given by:

$$1\triangleleft_T 2 \triangleleft_T 3, \quad 5 \triangleleft_T 4 \triangleleft_T 6 \triangleleft_T 3.$$
At this point, it is natural to ask for which binary trees $T$ we obtain a quasi-hereditary algebra $(A_n, \trianglelefteq_T)$.

Recall that, since $A_n$ is hereditary, $A_n$ is quasi-hereditary with respect to any partial order which is adapted to $A_n$. It turns out that any partial order $\trianglelefteq$ with respect to which $A_n$ is quasi-hereditary, is equivalent (that is, produces the same set of standard and costandard modules) to a partial order produced by a binary tree. Conversely, for any binary tree $T$, the order $\trianglelefteq_T$ makes $A_n$ quasi-hereditary. More precisely, we have the following.
\begin{proposition}\cite[Proposition~4.4]{FKR}
Let $n$ be a natural number and let $\mathcal{T}$ be the set of binary trees with $n$ vertices. Denote by $\mathcal{A}$ the set of adapted orders to $A_n$. For any partial order $\trianglelefteq$ in $\mathcal{A}$, denote by $\overline{\trianglelefteq}$ the equivalence class of $\trianglelefteq$ with respect to the relation $\sim$. Then, the map from $\mathcal{T}$ to $\faktor{\mathcal{A}}{\sim}$ defined by
$$T\mapsto \overline{\trianglelefteq}_T$$ is a bijection.
\end{proposition}
Now, given that any quasi-hereditary structure on $A_n$ may be associated to a binary tree, it is natural to describe this structure in terms of the associated binary tree.
The following proposition is stated in the proof of \cite[Proposition~4.4]{FKR}. For the convenience of the reader, we give a proof.
\begin{proposition}\label{proposition:standard and costandard modules over tree order}
Let $T$ be a binary search tree and let $\trianglelefteq_T$ be the associated adapted order to $A_n$. Then, we have the following.
\begin{enumerate}[(i)]
	\item The composition factors of the standard module $\Delta(i)$ are indexed by the labels of the vertices in the right subtree of the vertex labeled by $i$.
	\item The composition factors of the costandard module $\nabla(i)$ are indexed by the labels of the vertices in the left subtree of the vertex labeled by $j$.
\end{enumerate}
\end{proposition}
\begin{proof}
By definition, the standard module $\Delta(i)$ is a quotient of the indecomposable projective module $P(i)$. Since $P(i)=M(i,n)$, this implies that there exists an integer $s_i$, which satisfies $i\leq s_i\leq n$, such that $\Delta(i)=M(i,s_i)$. The composition factors of $\Delta(i)$ are then $L(i), \dots, L(s_i)$, all occuring with Jordan-Hölder multiplicity 1. Denote by $v_i$ the vertex labeled by $i$.

The composition factors $L(j)$ of $\Delta(i)$ must satisfy $v_j\trianglelefteq_T v_i$. By definition of $\trianglelefteq_T$, the vertices $v_j$ such that $v_j \triangleleft_T v_i$ are exactly the vertices of the left and right subtree of $v_i$.

 By construction of $\trianglelefteq_T$, the vertices in the left subtree of $v_i$ are labeled by integers less than $i$, which shows that the corresponding simple modules may not be composition factors of $\Delta(i)$. Since $\Delta(i)$ is the maximal quotient of $P(i)$ such that its composition factors $L(j)$ satisfy $v_j\trianglelefteq_T v_i$, we are done.
 
 The argument for the form of the costandard modules is similar. \qedhere
\end{proof}
For a vertex $v$ of $T$, denote by $\ell(v)$ and $r(v)$ the vertices immediately down and to the left or down and to the right of $v$, respectively. If such vertices do not exist, we write $\ell(v)=\emptyset$ or $r(v)=\emptyset$.
\begin{proposition}\label{proposition:extensions and homes between standard modules}
Let $T$ be a binary search tree and let $\trianglelefteq_T$ be the associated adapted order to $A_n$.
\begin{enumerate}[(i)]
	\item Suppose that $i$ labels the vertex $\ell(v)$ and that $j$ labels the vertex $v$. Then, we have $$\dim \operatorname{Ext}^1_{A_n}(\Delta(i),\Delta(j))=1.$$
	\item Suppose that $i$ labels the vertex $r(v)$ and that $j$ labels the vertex $v$. Then, we have $$\dim \operatorname{Hom}_{A_n}(\Delta(i),\Delta(j))=1.$$
\end{enumerate}
\end{proposition}
\begin{proof}
	Put $\Delta(i)=M(i,s_i)$ and $\Delta(j)=M(j, s_j)$.
	\begin{enumerate}[(i)]
		\item Note that the first vertex visited by the in-order algorithm, after visiting the right subtree of $\ell(v)$, is $v$. Therefore, by Proposition \ref{proposition:standard and costandard modules over tree order}, we have $j=s_i+1$. The condition of Proposition \ref{proposition:homs and ext between interval modules}, part (ii), is then
		$$i+1\leq s_i+1\leq s_i+1 \leq s_j,$$
		which is clearly satisfied.
		\item With the in-order algorithm, the vertex $r(v)$ is visited after $v$. Then, using Proposition \ref{proposition:standard and costandard modules over tree order}, we know that if $\Delta(i)=M(i,s_i)$, then $\Delta(j)=M(i+k,s_i)$, for some $k\leq s_i-i$. The condition of Proposition \ref{proposition:homs and ext between interval modules}, part (i), is then
		$$i\leq i+k \leq s_i \leq s_i,$$
		which is clearly satisfied.\qedhere
	\end{enumerate}
\end{proof}
\begin{lemma}\label{lemma:no hom from standard to left subtree and no ext from standard to right subtree}
	\begin{enumerate}[(i)]
		\item Let $v_i$ and $v_j$ be vertices labeled by $i$ and $j$, respectively. Assume that $v_j$ is in the left subtree of $v_i$. Then, we have
		$$\operatorname{Hom}_{A_n}(\Delta(i),\Delta(j))=\operatorname{Hom}_{A_n}(\Delta(j),\Delta(i))=0.$$
		\item Let $v_i$ and $v_j$ be vertices labeled by $i$ and $j$, respectively. Assume that $v_j$ is in the right subtree of $v_i$. Then, we have
		$$\operatorname{Ext}^1_{A_n}(\Delta(i),\Delta(j))=\operatorname{Ext}^1_{A_n}(\Delta(j),\Delta(i))=0.$$
		\item Let $v_i$ and $v_j$ be vertices, labeled by $i$ and $j$, respectively. Assume that $v_i$ is not in the subtree of $v_j$ and that $v_j$ is not in the subtree of $v_i$. Then, we have
		$$\operatorname{Hom}_{A_n}(\Delta(i),\Delta(j))=\operatorname{Hom}_{A_n}(\Delta(j),\Delta(i))=\operatorname{Ext}^1_{A_n}(\Delta(i),\Delta(j))=\operatorname{Ext}^1_{A_n}(\Delta(j),\Delta(i))=0.$$
	\end{enumerate}
\end{lemma}
\begin{proof} By the proof of Proposition \ref{proposition:standard and costandard modules over tree order}, we have $\Delta(i)=M(i, s_i)$ and $\Delta(j)=M(j,s_j)$, where the integers 
	$$i+1,i+2, \dots, s_i$$
	label the vertices in the right subtree of $v_i$ and the integers $$j+1,j+2,\dots, s_j$$ label the vertices in the right subtree of $v_j$.
	\begin{enumerate}[(i)]
		\item   According to the in-order algorithm, the integers $j+1,\dots, s_j$ are less than $i$. We immediately have $$\operatorname{Hom}_{A_n}(\Delta(i),\Delta(j))=0,$$
		since $v_j\triangleleft_T v_i$ and $A_n$ is quasi-hereditary. Moreover
		by Proposition \ref{proposition:homs and ext between interval modules}, part (ii), we have
		$$\operatorname{Hom}_{A_n}(\Delta(j),\Delta(i))=0,$$
		since the condition $i\leq j \leq s_i \leq s_j$ is not satisfied, because $j<i$.
		\item Since $v_j$ is in the right subtree of $v_i$, we have $s_j \leq s_i$. We immediately have
		$$\operatorname{Ext}_{A_n}^1(\Delta(i),\Delta(j))=0,$$
		since $v_j\triangleleft_T v_i$ and $A_n$ is quasi-hereditary. Moreover,
		by Proposition \ref{proposition:homs and ext between interval modules}, part (i), we have
		$$\operatorname{Ext}_{A_n}^1(\Delta(j),\Delta(i))=0,$$
		 since the condition $i+1\leq j\leq s_i+1\leq s_j$ is not satisfied, because $s_j\leq s_i$.

		\item This is immediate, since $A_n$ is quasi-hereditary and the vertices $v_i$ and $v_j$ are incomparable with respect to $\trianglelefteq_T$.  \qedhere
	\end{enumerate}
\end{proof}
\begin{example}
	Consider the following binary search tree, $T$, with vertices labeled according to the in-order algorithm.
	
	$$
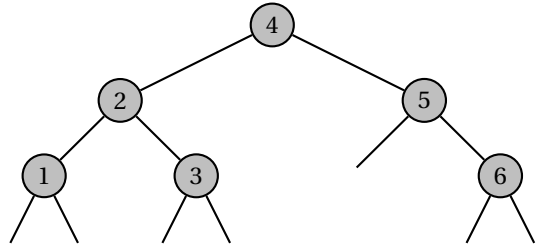
\begin{tikzpicture}
		\node(a) [shape=circle, draw, thick, fill=lightgray] at (0,0) {4};
		\node(b) [shape=circle, draw, thick, fill=lightgray] at (-2,-1) {2};
		\node(c) [shape=circle, draw, thick, fill=lightgray] at (2,-1) {5};
		\node(d) [shape=circle, draw, thick, fill=lightgray] at (-1,-2) {3};
		\node(e) [shape=circle, draw, thick, fill=lightgray] at (-3,-2) {1};
		\node(f)  at (1,-2) {};
		\node(g) [shape=circle, draw, thick, fill=lightgray] at (3,-2) {6};
		\node(h) at (-3.5,-3) {};
		\node(i)  at (-2.5,-3) {};
		\node(j)  at (-1.5,-3) {};
		\node(k)  at (-0.5,-3) {};
		\node(n)  at (2.5,-3) {};
		\node(o) at (3.5,-3) {};
		\draw[thick] (a) to (b);
		\draw[thick] (b) to (e);
		\draw[thick] (e) to (h);
		\draw[thick] (e) to (i);
		\draw[thick] (b) to (d);
		\draw[thick] (d) to (j);
		\draw[thick] (d) to (k);
		\draw[thick] (a) to (c);
		\draw[thick] (c) to (f);
		\draw[thick] (c) to (g);
		\draw[thick] (g) to (n);
		\draw[thick] (g) to (o);
	\end{tikzpicture}$$
	The partial order $\trianglelefteq_T$ is given by:
	
	$$1\triangleleft_T 2 \triangleleft_T 4,\quad 3\triangleleft_T 2 \triangleleft_T 4,\quad \textrm{and}\quad 6\triangleleft_T 5 \triangleleft_T 4.$$
	
	For the standard and costandard modules, we have
	\begin{align*}\Delta(1)& \cong L(1),\quad \Delta(2)\cong M(2,3),\quad\Delta(3)\cong L(3),\quad \Delta(4)\cong P(4), \quad \Delta(5)\cong P(6),\quad  \Delta(6)\cong L(6), \\
		\nabla(1)&\cong L(1),\quad \nabla(2)\cong M(1,2),\quad \nabla(3)\cong L(3),\quad \nabla(4)\cong M(1,4),\quad \nabla(5)\cong L(5),\quad \textrm{and}\quad \nabla(6)\cong L(6),\end{align*}
	which we see by applying Proposition \ref{proposition:standard and costandard modules over tree order}. Moreover, we have non-split short exact sequences
	$$\Delta(2)\hookrightarrow I(3)\twoheadrightarrow \Delta(1)\quad \textrm{and}\quad \Delta(4) \hookrightarrow P(2)\twoheadrightarrow \Delta(2),$$
	giving us the extensions guaranteed by part (i) of Proposition \ref{proposition:extensions and homes between standard modules}. 
	
	Next, we see there are natural monomorphisms $\Delta(3)\hookrightarrow \Delta(2)$ and $\Delta(6)\hookrightarrow\Delta(5)\hookrightarrow\Delta(4),$ giving us the homomorphisms guaranteed by part (ii) of Proposition \ref{proposition:extensions and homes between standard modules}.
	
	Lastly, we have
	\begin{align*}
		\dim \operatorname{Hom}_{A_n}(\Delta(1),\Delta(2))&=\dim \operatorname{Hom}_{A_n}(\Delta(2),\Delta(1))=\dim \operatorname{Hom}_{A_n}(\Delta(1),\Delta(3))=\dim \operatorname{Hom}_{A_n}(\Delta(3),\Delta(1))=0;\\
		\dim \operatorname{Hom}_{A_n}(\Delta(2),\Delta(4))&=\dim \operatorname{Hom}_{A_n}(\Delta(4),\Delta(2))=\dim \operatorname{Hom}_{A_n}(\Delta(1),\Delta(4))=\dim \operatorname{Hom}_{A_n}(\Delta(4),\Delta(1))=0;\\
		\dim \operatorname{Hom}_{A_n}(\Delta(2),\Delta(5))&=\dim\operatorname{Hom}_{A_n}(\Delta(5),\Delta(2)) =\dim \operatorname{Hom}_{A_n}(\Delta(2),\Delta(6))=\dim \operatorname{Hom}_{A_n}(\Delta(6),\Delta(2))=0;\\
		\dim \operatorname{Hom}_{A_n}(\Delta(1),\Delta(5))&=\dim \operatorname{Hom}_{A_n}(\Delta(5),\Delta(1))=\dim \operatorname{Hom}_{A_n}(\Delta(1),\Delta(6))=\dim \operatorname{Hom}_{A_n}(\Delta(6),\Delta(1))=0,
	\end{align*}
	as prescribed by Lemma \ref{lemma:no hom from standard to left subtree and no ext from standard to right subtree}.
\end{example}
\begin{lemma} \label{lemma:multiplication map in right subtree is zero}
		Let $v, r(v)$ and $\ell(r(v))$ be vertices labeled by $i, j$ and $k$, respectively. Then, the multiplication map $$\operatorname{Hom}_{A_n}( \Delta(j),\Delta(i)) \times \operatorname{Ext}^1_{A_n}(\Delta(k),\Delta(j))\to \operatorname{Ext}_{A_n}^1(\Delta(k),\Delta(i))$$ is the zero map.
\end{lemma}
\begin{proof}
	 Note that the vertices $i,j$ and $k$ are configured in the following way:
	$$
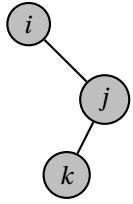
\begin{tikzpicture}
		\node(a)[shape=circle,draw, fill=lightgray, thick] at (0,0) {$i$};
		\node(b) [shape=circle,draw, fill=lightgray, thick] at (1,-1) {$j$};
		\node(c) [shape=circle,draw, fill=lightgray, thick] at (0.5, -2) {$k$};
		\draw[thick] (a) to (b);
		\draw[thick] (b) to (c);
	\end{tikzpicture}$$
The multiplication map in question produces an extension in the space $\operatorname{Ext}_{A_n}^1(\Delta(k),\Delta(i))$. But since $\ell(r(v))$ is a vertex in the right subtree of $v$, this space is zero, by part (ii) of Lemma \ref{lemma:no hom from standard to left subtree and no ext from standard to right subtree}.
\end{proof}
\begin{lemma}\label{lemma:multiplication in left subtree is non-zero}
	Let $v$, $\ell(v)$ and $r(\ell(v))$ be vertices labeled by $i,j$ and $k$, respectively. Then, the multiplication map
	$$\operatorname{Ext}_{A_n}^1(\Delta(j),\Delta(i))\times \operatorname{Hom}_{A_n}(\Delta(k),\Delta(j))\to \operatorname{Ext}_{A_n}^1(\Delta(k),\Delta(i))$$ is non-zero.
\end{lemma}
\begin{proof}
We start by observing that the vertices labeled by $i$, $j$ and $k$ are configured in the following way:
$$
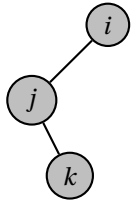
\begin{tikzpicture}
	\node(a)[shape=circle,draw, fill=lightgray, thick] at (0,0) {$i$};
	\node(b) [shape=circle,draw, fill=lightgray, thick] at (-1,-1) {$j$};
	\node(c) [shape=circle,draw, fill=lightgray, thick] at (-0.5, -2) {$k$};
	\draw[thick] (a) to (b);
	\draw[thick] (b) to (c);
\end{tikzpicture}$$
According to the in-order algorithm, we have $j<k<i$. Note that, since $A_n$ is hereditary, a minimal projective resolution of a standard module $\Delta(x)$ is of the form
$$\xymatrix{\ker p_x \ar[r] & P(x) \ar[r]^-{p_x} & \Delta(x)}.$$
Consider the following picture:
$$\xymatrix{
&\ker p_k \ar[r] \ar[d]^-f  &P(k) \ar[r]^-{p_k} \ar[d]& \Delta(k) \\
&\ker p_j \ar[r] \ar[d]^-g &P(j) \ar[r]^-{p_j} & \Delta(j) \\
\ker p_i \ar[r] &P(i) \ar[r]^-{p_k} & \Delta(i)
}$$
An extension from $\Delta(k)$ to $\Delta(i)$ is represented by a chain map, which has one (possibly) non-zero component, namely the map $g\circ f$. Now, assume that  $\Delta(k)=M(k,s_k)$ and $\Delta(i)=M(i, s_i)$. Then, we have $\ker p_k=P(s_k+1)$ and $\ker p_i=P(s_i+1)$. Since $k<i$ and $s_k<s_i$, we have
$\operatorname{Hom}_{A_n}(P(k),P(i))=\operatorname{Hom}_{A_n}(\ker p_k, \ker p_i)=0.$ Therefore, the chain map, if it is non-zero, cannot be null-homotopic. By construction, $f$ is the unique (up to a scalar) non-zero map from $M(s_k+1, n)$ to $M(s_j+1,n)$. Next, note that in our configuration, the first vertex visited by the in-order algorithm, after visiting the entire right subtree of $\ell(v)$, is $v$. This implies that $s_j+1=i$, so that the map $g$ in the above picture is some scalar multiple of the identity homomorphism on $P(i)$. This shows that the composition $g\circ f$, and hence the chain map representing our extension, is non-zero.
\end{proof}
For two interval modules $M(i_1, j_1)$ and $M(i_2, j_2)$, the space $\operatorname{Hom}_{A_n}(M(i_1, j_1), M(i_2, j_2))$ is one-dimensional if it is non-zero. Fix a basis of this space, corresponding to the following homomorphism of representations:

$$\xymatrix{
K & \dots \ar[l] & K \ar[l] \ar[d]^1& \dots \ar[l] \ar[d]^1& K \ar[l] \ar[d]^1\\
& & K & \dots \ar[l] & K \ar[l] & \dots \ar[l] & K\ar[l]
}$$
Note that this choice of basis also gives an obvious choice of basis for the spaces $\operatorname{Ext}_{A_n}^1(M(i_1, j_1), M(i_2, j_2))$.
Suppose now that we have chosen such basis vectors
$$y\in \operatorname{Hom}_{A_n}(\Delta(i), \Delta(j));\quad x\in \operatorname{Ext}_{A_n}^1(\Delta(j), \Delta(k));\quad z\in \operatorname{Ext}_{A_n}^1(\Delta(i),\Delta(k));$$
and that $i,j$ and $k$ are as in the statement of Lemma \ref{lemma:multiplication in left subtree is non-zero}. Then, we have $xy=z$.
\begin{theorem}\label{theorem:ext-algebra of A_n}
	Let $T$ be a binary search tree with $n$ vertices, with the vertices labeled according to the in-order algorithm, and let $\trianglelefteq_T$ be the associated adapted order to $A_n$. Construct a quiver $Q$ in the following way:
	\begin{enumerate}[(i)]
		\item The vertices of $Q$ are $1,2,\dots, n$.
		\item We draw an edge between $i$ and $j$ if and only if the integers $i$ and $j$ label either a set of vertices $\{v, \ell(v)\}$ or $\{v, r(v)\}$. If $i$ labels $\ell(v)$ and $j$ labels $v$, the orientation of the edge is from $\ell(v)$ to $v$, of degree 1, and is denoted by $\varepsilon_i^j$. If $i$ labels $r(v)$ and $j$ labels $v$, the orientation of the edge is from $r(v)$ to $v$, of degree 0, and is denoted by $f_i^j$.
	\end{enumerate}
Let $I\subset KQ$ be the ideal generated by the following elements:
\begin{enumerate}[($\bullet$)]
	\item $f_j^k \varepsilon_i^j$, for all $i,j$ and $k$;
	\item $\varepsilon_j^k \varepsilon_i^j$, for all $i,j$ and $k$;
\end{enumerate}
Then, there is an isomorphism of graded algebras $\faktor{KQ}{I}\cong \operatorname{Ext}_{A_n}^\ast(\Delta,\Delta)$.
\end{theorem}
\begin{example}
	Consider the binary search tree
	$$
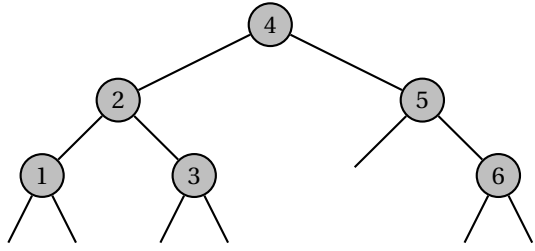
\begin{tikzpicture}
		\node(a) [shape=circle, draw, thick, fill=lightgray] at (0,0) {4};
		\node(b) [shape=circle, draw, thick, fill=lightgray] at (-2,-1) {2};
		\node(c) [shape=circle, draw, thick, fill=lightgray] at (2,-1) {5};
		\node(d) [shape=circle, draw, thick, fill=lightgray] at (-1,-2) {3};
		\node(e) [shape=circle, draw, thick, fill=lightgray] at (-3,-2) {1};
		\node(f)  at (1,-2) {};
		\node(g) [shape=circle, draw, thick, fill=lightgray] at (3,-2) {6};
		\node(h) at (-3.5,-3) {};
		\node(i)  at (-2.5,-3) {};
		\node(j)  at (-1.5,-3) {};
		\node(k)  at (-0.5,-3) {};
		\node(n)  at (2.5,-3) {};
		\node(o) at (3.5,-3) {};
		\draw[thick] (a) to (b);
		\draw[thick] (b) to (e);
		\draw[thick] (e) to (h);
		\draw[thick] (e) to (i);
		\draw[thick] (b) to (d);
		\draw[thick] (d) to (j);
		\draw[thick] (d) to (k);
		\draw[thick] (a) to (c);
		\draw[thick] (c) to (f);
		\draw[thick] (c) to (g);
		\draw[thick] (g) to (n);
		\draw[thick] (g) to (o);
	\end{tikzpicture}$$
Then, the corresponding quiver $Q$ is
$$
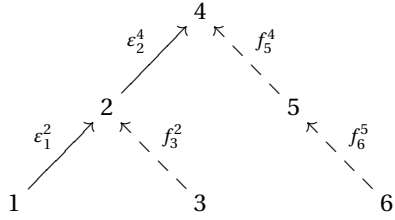
\xymatrix{
	& & 4\\
	& 2\ar[ru]^-{\varepsilon_2^4} & & 5\ar@{-->}[lu]_-{f_5^4}\\ 
1 \ar[ru]^-{\varepsilon_1^2}	& & 3 \ar@{-->}[lu]_-{f_3^2} & &6\ar@{-->}[lu]_-{f_6^5}
}$$
and the ideal $I=\langle \varepsilon_2^4\varepsilon_1^2\rangle$. The full arrows, $\varepsilon_1^2$ and $\varepsilon_2^4$, are of degree 1, while the dashed arrows, $f_3^2, f_5^4$ and $f_6^5$ are of degree 0.
\end{example}
\begin{proof}[Proof of Theorem \ref{theorem:ext-algebra of A_n}]
Fix basis vectors $\varphi_i^j\in \operatorname{Hom}_{A_n}(\Delta(i),\Delta(j))$ and $e_i^j\in \operatorname{Ext}_{A_n}^1(\Delta(i),\Delta(j))$ of all non-zero spaces of the form $\operatorname{Hom}_{A_n}(\Delta(i),\Delta(j))$ and $\operatorname{Ext}_{A_n}^1(\Delta(i),\Delta(j))$, as discussed prior to Theorem \ref{theorem:ext-algebra of A_n}. Define a map $\Phi:\faktor{KQ}{I}\to \operatorname{Ext}_{A_n}^\ast(\Delta,\Delta)$ by 
$$\varepsilon_i^j \mapsto e_i^j,\quad f_i^j\mapsto \varphi_i^j, \quad e_i\mapsto 1_{\Delta(i)},$$ and extend by linearity. By applying Lemma \ref{lemma:multiplication map in right subtree is zero} and Lemma \ref{lemma:no hom from standard to left subtree and no ext from standard to right subtree}, along with the fact that there are no extensions of degree greater than 1, since $A_n$ is hereditary, we see that $I\subset \ker \Phi$, so that $\Phi$ is well-defined.

To check that $\Phi$ is a homomorphism of algebras, it is enough to check compatibility of $\Phi$ with products of the form $\varepsilon_j^k f_i^j$. We have:
\begin{align*}
	\Phi(\varepsilon_j^k f_i^j)&=\Phi(\varepsilon_i^k)=e_i^k=e_j^k\varphi_i^j=\Phi(\varepsilon_j^k)\Phi(f_i^j).
\end{align*}
As the image of $\Phi$ contains bases of $\operatorname{Hom}_{A_n}(\Delta(i),\Delta(j))$ and $\operatorname{Ext}_{A_n}^1(\Delta(i),\Delta(j))$ for all $i$ and $j$, $\Phi$ is surjective. 

Next, we compare dimensions. Consider the degree 0 part of the space $e_j \faktor{KQ}{I}e_i$. The dimension of this space is equal to the number of paths of degree 0 from $i$ to $j$ in $Q$. Clearly, such a path must be a product of the form $f_{k_s}^j f_{k_{s-1}}^j \dots f_{k_1}^{k_2}f_i^{k_1}$. Since vertices in $Q$ have out-degree at most 1, it follows that such a path is unique. So the dimension of the degree 0 part $e_j\faktor{KQ}{I}e_i$ is either 0 or 1. By construction of $Q$, it is now clear that this dimension coincides with $\dim \operatorname{Hom}_{A_n}(\Delta(i),\Delta(j))$. Next, consider the degree 1 part of the space $e_j\faktor{KQ}{I}e_i$. Similarly, the dimension of this space is either 0 or 1, depending on the number of paths from $i$ to $j$ in $Q$. Such a path is either of the form $\varepsilon_i^j$ or $\varepsilon_k^j f_i^k$. Again, by construction of $Q$, we see that this dimension coincides with $\operatorname{Ext}_{A_n}^1(\Delta(i),\Delta(j))$. Therefore, $\Phi$ is a surjective linear map between vector spaces of the same dimension, hence an isomorphism.
\end{proof}
\subsection{Ringel dual}
We recall that, for any quasi-hereditary algebra $A$, there exists a module $\mathtt{T}$, called the characteristic tilting module, whose additive closure equals $\mathcal{F}(\Delta)\cap \mathcal{F}(\nabla)$. The characteristic tilting module decomposes as
$$\mathtt{T}=\bigoplus_{i=1}^n T(i),$$
where each $T(i)$ is indecomposable. The Ringel dual is then, by definition, $R(A)=\operatorname{End}(\mathtt{T})^{\operatorname{op}}$.
In this section, we compute quiver and relations of the Ringel dual of $(A_n,\trianglelefteq_T)$ using the associated binary search tree. The following statement is part of the proof of Theorem~4.6 in \cite{FKR}. For the convenience of the reader, we give a proof.

\begin{proposition}\label{proposition:comp factors of characteristic tilting modules}
Let $v$ be a vertex labeled by $i$. The indecomposable summand $T(i)$ of the characteristic tilting module is the interval module $M(t_i, s_i)$, where $t_i$ is the least integer labelling a vertex in the left subtree of $v$ and $s_i$ is the greatest integer labelling a vertex in the right subtree of $v$.
\end{proposition}
\begin{proof}
To prove that $T(i)=M(t_i, s_i)$, it suffices to show that $\Delta(i)$ is isomorphic to a submodule of $T(i)$, that the cokernel of this monomorphism admits a filtration by standard modules, and that $T(i)$ is contained in $\mathcal{F}(\Delta)\cap \mathcal{F}(\nabla)$ \cite[Proposition~2]{Ringel1991}. We proceed by induction on the size of the subtree of $v$. The basis of the induction is clear, since if $\ell(v)=r(v)=\emptyset$, we have $\Delta(i)=\nabla(i)=L(i)=T(i)$. Since $t_i\leq i$, it is clear that $\Delta(i)$ embeds into $M(t_i, s_i)$. Clearly, the cokernel of this embedding is the module $M(t_i, i-1)$.

Suppose that $j$ labels the vertex $\ell(v)$. The composition factors of $M(t_i, i-1)$ are labeled by the vertices in the subtree of $\ell(v)$, so by the induction hypothesis, $M(t_i, i-1)=T(j)$. In other words, we have a short exact sequence
$$0\to \Delta(i) \hookrightarrow M(t_i, s_i) \twoheadrightarrow T(j)\to 0.$$
 Since $T(j)\in \mathcal{F}(\Delta)\cap \mathcal{F}(\nabla)$, we have $M(t_i, s_i)\in \mathcal{F}(\Delta)$. By a similar argument, we see that there is a short exact sequence
 $$0 \to T(k) \hookrightarrow T(i) \twoheadrightarrow \nabla(i)\to 0,$$
 where $k$ labels the vertex $r(v)$. Then, since $T(k)\in \mathcal{F}(\Delta)\cap \mathcal{F}(\nabla)$, we get that $M(t_i, s_i)\in \mathcal{F}(\nabla)$. Now, $M(t_i, s_i)$ is a module such that $\Delta(i)$ embeds into it, the cokernel of this embedding is contained in $\mathcal{F}(\Delta)$, and $M(t_i, s_i)\in \mathcal{F}(\Delta)\cap \mathcal{F}(\nabla)$. This confirms that $M(t_i, s_i)=T(i)$.
\end{proof}
\begin{corollary}\label{corollary:proposition:comp factors of characteristic tilting modules}
	Let the vertices $v$, $\ell(v)$ and $r(v)$ be labeled by $i, j$ and $k$, respectively. Then, for any $v$, there is an epimorphism $p_i^j:T(i)\twoheadrightarrow T(j)$
	and a monomorphism $q_k^i:T(k)\hookrightarrow T(i).$
\end{corollary}
\begin{proposition} \label{proposition: homs between summands of tilting module}
	Let the vertices $\ell(\ell(v)), \ell(v),v, r(v)$ and $r(r(v))$ be labeled by labeled by $x, j, i, k$ and $y$, respectively. Let $p_i^j, p_j^x, q_k^i$ and $q_y^k$ be the homomorphisms from Corollary \ref{corollary:proposition:comp factors of characteristic tilting modules}. Then, we have:
	\begin{enumerate}[(i)]
		\item $p_j^x p_i^j\neq 0$ and $q_k^i q_y^k \neq 0$;
		\item $p_i^j q_k^i=0;$
		\item $\operatorname{Hom}_{A_n}(T(j), T(i))=0$;
		\item $\operatorname{Hom}_{A_n}(T(i), T(k))=0.$
	\end{enumerate}
\end{proposition}
\begin{proof}
	The vertices are configured in the following way:
		$$
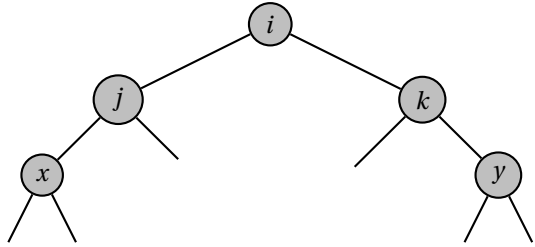
\begin{tikzpicture}
		\node(a) [shape=circle, draw, thick, fill=lightgray] at (0,0) {$i$};
		\node(b) [shape=circle, draw, thick, fill=lightgray] at (-2,-1) {$j$};
		\node(c) [shape=circle, draw, thick, fill=lightgray] at (2,-1) {$k$};
		\node(e) [shape=circle, draw, thick, fill=lightgray] at (-3,-2) {$x$};
		\node(f)  at (1,-2) {};
		\node(g) [shape=circle, draw, thick, fill=lightgray] at (3,-2) {$y$};
		\node(h) at (-3.5,-3) {};
		\node(i)  at (-2.5,-3) {};
		\node(j)  at (-1.5,-3) {};
		\node(k)  at (-0.5,-3) {};
		\node(n)  at (2.5,-3) {};
		\node(o) at (3.5,-3) {};
		\draw[thick] (a) to (b);
		\draw[thick] (b) to (e);
		\draw[thick] (e) to (h);
		\draw[thick] (e) to (i);
		\draw[thick] (b) to (d);
		\draw[thick] (a) to (c);
		\draw[thick] (c) to (f);
		\draw[thick] (c) to (g);
		\draw[thick] (g) to (n);
		\draw[thick] (g) to (o);
	\end{tikzpicture}$$
\begin{enumerate}[(i)]
	\item This is immediate, as $p_j^xp_i^j$ is the composition of two epimorphisms and $q_k^iq_y^k$ is the composition of two monomorphisms.	
	\item Applying Proposition \ref{proposition:comp factors of characteristic tilting modules}, we see that $T(k)$ and $T(j)$ have no common composition factors, which implies that $\operatorname{Hom}_{A_n}(T(k),T(j))=0$, yielding the statement.
	 \item Put $T(j)=M(t_j, s_j)$ and $T(i)=M(t_i, s_i)$, where $t_j=t_i$ and $s_j< s_i$. Then, using Proposition \ref{proposition:homs and ext between interval modules}, the condition for $\operatorname{Hom}_{A_n}(T(j), T(i))$ being non-zero is
	 $$t_i\leq t_j \leq s_i \leq s_j,$$
	 which is not satisfied, since $s_j < s_i$.
	 \item Similar to (iii). \qedhere
\end{enumerate}
With these observations established, we are ready to describe the Ringel dual $R(A_n)$.
\end{proof}
\begin{theorem}\label{theorem:ringel dual of A_n}
	Let $T$ be a binary search tree with $n$ vertices, with the vertices labeled according to the in-order algorithm, and let $\trianglelefteq_T$ be the associated adapted order to $A_n$. Construct a quiver $Q$ in the following way:
	\begin{enumerate}[(i)]
		\item The vertices of $Q$ are $1,\dots, n$.
		\item We draw an edge between $i$ and $j$ if and only if the integers $i$ and $j$ label either a set of vertices $\{v,\ell(v)\}$ or $\{v,r(v)\}$. If $i$ labels $v$ and $j$ labels $\ell(v)$, the edge is denoted by $f_i^j$ and its orientation is from $i$ to $j$. If $i$ labels $v$ and $j$ labels $r(v)$, the edge is denoted by $g_j^i$ and its orientation is from $j$ to $i$.
	\end{enumerate}
Let $I\subset KQ$ be the ideal generated by the elements $f_i^j g_k^i$, for all $i,j$ and $k$. Then, there is an isomorphism of algebras $\faktor{KQ}{I}\cong\operatorname{End}(\mathtt{T})$, and consequently, $R(A_n)\cong \faktor{KQ}{I}^{\operatorname{op}}$.
\end{theorem}
\begin{example}
		Consider the binary search tree
	$$
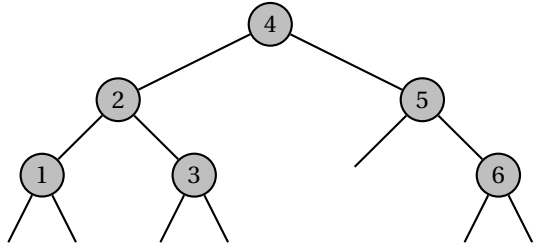
\begin{tikzpicture}
		\node(a) [shape=circle, draw, thick, fill=lightgray] at (0,0) {4};
		\node(b) [shape=circle, draw, thick, fill=lightgray] at (-2,-1) {2};
		\node(c) [shape=circle, draw, thick, fill=lightgray] at (2,-1) {5};
		\node(d) [shape=circle, draw, thick, fill=lightgray] at (-1,-2) {3};
		\node(e) [shape=circle, draw, thick, fill=lightgray] at (-3,-2) {1};
		\node(f)  at (1,-2) {};
		\node(g) [shape=circle, draw, thick, fill=lightgray] at (3,-2) {6};
		\node(h) at (-3.5,-3) {};
		\node(i)  at (-2.5,-3) {};
		\node(j)  at (-1.5,-3) {};
		\node(k)  at (-0.5,-3) {};
		\node(n)  at (2.5,-3) {};
		\node(o) at (3.5,-3) {};
		\draw[thick] (a) to (b);
		\draw[thick] (b) to (e);
		\draw[thick] (e) to (h);
		\draw[thick] (e) to (i);
		\draw[thick] (b) to (d);
		\draw[thick] (d) to (j);
		\draw[thick] (d) to (k);
		\draw[thick] (a) to (c);
		\draw[thick] (c) to (f);
		\draw[thick] (c) to (g);
		\draw[thick] (g) to (n);
		\draw[thick] (g) to (o);
	\end{tikzpicture}$$
	Then, the corresponding quiver $Q$ is
	$$
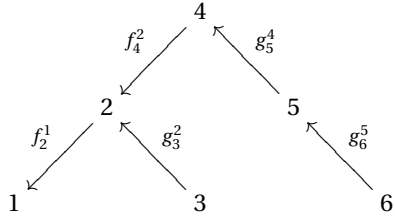
\xymatrix{
		& & 4 \ar[ld]_-{f_4^2}\\
		& 2 \ar[ld]_-{f_2^1}& & 5 \ar[lu]_-{g_5^4}\\ 
		1 	& & 3  \ar[lu]_-{g_3^2}& &6 \ar[lu]_-{g_6^5}
	}$$
and the ideal $I\subset KQ$ is generated by the elements $f_4^2 g_5^4$ and $f_2^1 g_3^2$.
	\end{example}
\begin{proof}[Proof of Theorem~\ref{theorem:ringel dual of A_n}]
	Recall the discussion prior to Theorem \ref{theorem:ext-algebra of A_n}, which implies that for composable homomorphisms between indecomposable summands of $\mathtt{T}$, we may choose basis vectors 
	\begin{align*}
		x \in \operatorname{Hom}_{A_n}(T(i), T(j));\quad y\in \operatorname{Hom}_{A_n}(T(j), T(k));\quad z\in \operatorname{Hom}_{A_n}(T(i), T(k))
	\end{align*}
such that $xy=z$. Define a map $\tilde{\Psi}: KQ\to \operatorname{End}(\mathtt{T})$ by
$$f_i^j \mapsto p_i^j,\quad g_i^j \mapsto q_i^j,\quad e_i \mapsto 1_{T(i)}$$
and extend by linearity. Using Proposition \ref{proposition: homs between summands of tilting module}, part (ii), we see that $I\subset \ker \tilde{\Psi}$, so that $\tilde\Psi$ factors uniquely through the quotient $\faktor{KQ}{I}$. This gives us a well-defined homomorphism of algebras $\Psi:\faktor{KQ}{I}\to \operatorname{End}(\mathtt{T})$. Since the homomorphisms $p_i^j, q_k^i$ and $1_{T(i)}$ generate $\operatorname{End}(\mathtt{T})$, $\Psi$ is surjective. It is now easy to see that $\dim \faktor{KQ}{I}=\dim \operatorname{End}(\mathtt{T})$, so that $\Psi$ is a surjective linear map between spaces of the same dimension, hence an isomorphism.
	\end{proof}
\subsection{Regular exact Borel subalgebras of $A_n$}
In this section, we investigate when $(A_n, \trianglelefteq_T)$ admits a regular exact Borel subalgebra. Recall the following.

\begin{definition}\cite[Definition~3.4]{konig1995exact, bkk}\label{definition:exact borel subalg}
	Let $(\Lambda,\trianglelefteq)$ be a quasi-hereditary algebra with $n$ simple modules, up to isomorphism. Then, a subalgebra $B\subset \Lambda$ is called an \emph{exact Borel subalgebra} provided that
	\begin{enumerate}[(i)]
		\item $B$ also has $n$ simple modules up to isomorphism and $(B,\trianglelefteq)$ is quasi-hereditary with simple standard modules,
		\item the functor $\Lambda\otimes_B\blank$ is exact, and
		\item there are isomorphisms $\Lambda\otimes_B L_B(i)\cong \Delta_\Lambda(i).$
	\end{enumerate}
	If, in addition, the map $\operatorname{Ext}_B^k(L_B(i),L_B(j))\to \operatorname{Ext}_\Lambda^k(\Lambda\otimes_B L_B(i), \Lambda\otimes_B L_B(j))$ induced by the functor $\Lambda\otimes_B\blank$ is an isomorphism for all $k\geq 1$ and $i,j\in\{1,\dots, n\}$, $B\subset \Lambda$ is called a \emph{regular exact Borel subalgebra}. If this map is an isomorphism for all $k\geq 2$ and an epimorphism for $k=1$, $B\subset \Lambda$ is called a \emph{homological} Borel exact Borel subalgebra.
\end{definition}

\begin{theorem}\label{theorem:when does KA_n have regular exact borel}\cite[Theorem~D, part 5]{CONDE21}
	Every quasi-hereditary algebra equivalent to $(A_n, \trianglelefteq_T)$ has a regular exact Borel subalgebra if and only if $\operatorname{rad}\Delta(i)$ is contained in $\mathcal{F}(\nabla)$, for all $i\in\{1,\dots, n\}$.
	\end{theorem}
\begin{proposition}\label{proposition:A_n has a regular exact borel}
	$(A_n, \trianglelefteq_T)$ has a regular exact Borel subalgebra.
\end{proposition}
\begin{proof}
	Let the vertex $v$ be labeled by $i$ and consider $\Delta(i)=M(i, s_i)$. Then, $\operatorname{rad}\Delta(i)=M(i+1, s_i)$. Let the vertex $r(v)$ be labeled by $k$. According to the in-order algorithm, the first vertex visited immmediately after $v$ is the left-most vertex in the subtree of $r(v)$. This shows, in our earlier notation, that $t_k=i+1$. This fact implies that
	$$M(i+1, s_i)=M(t_k, s_i)=M(t_k, s_k)=T(k),$$
	 according to the proof of Proposition \ref{proposition:comp factors of characteristic tilting modules}. Since $T(k)\in \mathcal{F}(\nabla)$, the statement follows from Theorem \ref{theorem:when does KA_n have regular exact borel}.
\end{proof}
Thanks to the obervations made earlier in the current section, it is fairly easy to find the quiver, $Q_B$, of a regular exact Borel subalgebra $B\subset A_n$. Recall that, if $B\subset A_n$ is a regular exact Borel subalgebra, then there are isomorphisms of vector spaces
$$\operatorname{Ext}_B^k(L(i),L(j))\cong \operatorname{Ext}_{A_n}^k(\Delta(i),\Delta(j)),$$
for all $k\geq 1$ and $i,j\in\{1,\dots, n\}$. In particular, the spaces above have the same dimension. In our setup, these spaces are zero for $k> 1$. It is well-known that
$$\dim \operatorname{Ext}_B^1(L(i),L(j))= \textrm{number of arrows from $i$ to $j$ in $Q_B$}.$$
This means that we can determine the quiver $Q_B$ by computing extensions between standard modules over $A_n$. To this end, we apply Theorem \ref{theorem:ext-algebra of A_n}, resulting in the following. 
\begin{proposition}\label{proposition:quiver of regular exact borel}
	Let $B\subset A_n$ be a regular exact Borel subalgebra. Then, there is an isomorphism of algebras $B\cong KQ_B$, where $Q_B$ is the following quiver.
	\begin{enumerate}[(i)]
		\item The vertex set of $Q_B$ is $\{1,\dots, n\}$.
		\item There is an arrow $i\to j$ if and only if $\dim \operatorname{Ext}_{A_n}^1(\Delta(i),\Delta(j))=1$.
	\end{enumerate}
\end{proposition}
Next, we wish to determine how to realize $B$ as a subalgebra of $A_n$. In this endeavour, we are aided by the following results.
\begin{theorem}[Wedderburn-Mal'tsev]\cite{Wedderburn, Mal'tsev}\label{theorem:wedderburn}
Let $A$ be a finite-dimensional algebra over an algebraically closed field. Then, there is a semi-simple subalgebra $S$ of $A$ such that
$$A=\operatorname{rad}A\oplus S$$
as a vector space. Moreover, if $T$ is another semisimple subalgebra of $A$ such that $A=\operatorname{rad}A\oplus T$, then there is an inner automorphism $\sigma$ of $A$ such that $\sigma(S)=T$.
\end{theorem}
We remark that if $S$ is a semi-simple subalgebra as in the theorem and $T$ is another semi-simple subalgebra with $\dim_KS=\dim_KT$, then necessarily $A=\operatorname{rad}A\oplus T$, as any semi-simple subalgebra intersects $\operatorname{rad}A$ trivially. Consequently, $S$ and $T$ are conjugate.

Let $A=KQ/I$ be the path algebra of some quiver modulo an admissible ideal and let $S\subset A$ be the subalgebra $S=\langle e_1,\dots, e_n\rangle$. Then, $S$ is semi-simple and $A=\operatorname{rad}A\oplus S$. If $\{d_1,\dots, d_n\}$ is some other complete set of primitive orthogonal idempotents, then $T=\langle d_1,\dots, d_n\rangle$ is another semi-simple subalgebra of $A$ with $\dim_KS=\dim_K T$, so $S$ and $T$ must be conjugate. We conclude that for any two complete sets of primitive orthogonal idempotents in $A$, there is an inner automorphism of $A$ mapping one to the other.
\begin{theorem}\cite[Proposition~3.4]{Zhangunique}\label{theorem: inner auto preserves borel}
Let $A$ be a finite-dimensional quasi-hereditary algebra and let $B\subset A$ be an exact Borel subalgebra. Then, $C=\sigma(B)$ is again an exact Borel subalgebra.
\end{theorem}
\begin{theorem}\cite[Theorem~3.6]{Zhangunique}\label{theorem: uniqueness of borel} Let $A$ be a finite-dimensional quasi-hereditary algebra and let $B\subset A$ be an exact Borel subalgebra. Then, $B\subset A$ is unique up to an inner automorphism of $A$, that is, if $B^\prime \subset A$ is another exact Borel subalgebra, then $B$ and $B^\prime$ are conjugate.
\end{theorem}
Let $\{d_1,\dots, d_n\}$ be a complete set of primitive orthogonal idempotents in $B$. For any complete set of primitive orthogonal idempotents $\{f_1,\dots, f_n\}$ of $A_n$, there is an inner automorphism $\sigma$ of $A_n$ mapping $\{d_1,\dots, d_n\}$ to $\{f_1,\dots, f_n\}$. But by Theorem \ref{theorem: inner auto preserves borel}, $\sigma(B)$ is an exact Borel subalgebra of $A_n$, and by construction it contains the idempotents $\{f_1,\dots, f_n\}$. We conclude that for any complete set of primitive orthogonal idempotents $\{f_1,\dots, f_n\}$ in $A_n$, there is an exact Borel subalgebra $C$ containing them. In particular, there is an exact Borel subalgebra containing the idempotents $e_1,\dots, e_n$.
\begin{example}
Consider $A_2$, the path algebra of $\xymatrix{1\ar[r]^-\alpha & 2}$ and let $B\subset A_2$ be a regular exact Borel subalgebra.
Fix the order $1\triangleleft 2$. Then we have
$$\Delta(1)\cong L(1)\quad\textrm{and}\quad \Delta(2)\cong L(2).$$
Therefore, the quiver of $B$ coincides with the quiver of $A_2$, and the (unique) regular exact Borel subalgebra is $B=A_2$. If we instead fix the order  $2\triangleleft 1$, we have
$$\Delta(1)\cong P(1)\quad\textrm{and}\quad \Delta(2)\cong P(2).$$
Since there are no non-split extensions between projective modules, the quiver of $B$ has no arrows. Thus, we have $B\cong K\times K$. The obvious choice for $B$ is then the subalgebra $B=\langle e_1,e_2\rangle$. Suppose $xe_1+ye_2+z\alpha\in A_2$ is an idempotent. Then we have:
\begin{align*}
	(xe_1+ye_2+z\alpha)^2=x^2e_1+y^2e_2+z(x+y)\alpha=xe_1+ye_2+z\alpha
\end{align*}
This implies that all possible complete sets of primitive orthogonal idempotents are $\{e_1+z\alpha, e_2-z\alpha\}$, where $z\in K$ is some scalar. Thus, we obtain exact Borel subalgebras $B_z=\langle e_1+z\alpha, e_2-z\alpha\rangle$. Since $z\alpha \in \operatorname{rad}A_2$, the element $1-z\alpha$ is invertible (with inverse $1+z\alpha$). We see that
$$(1+z\alpha)e_1(1-z\alpha)=1+z\alpha\quad \textrm{and} \quad (1+z\alpha)e_2(1-z\alpha)=1-z\alpha,$$
so for all $z\neq 0$, the subalgebras $B$ and $B_z$ are conjugate.
\end{example}
 Next, we want to argue as in the discussion after Theorem \ref{theorem: uniqueness of borel} to conclude that there is always a regular exact Borel subalgebra containing the idempotents $e_1,\dots, e_n$. That argument relies on \cite[Proposition~3.4]{Zhangunique}, which states that inner automorphisms of quasi-hereditary algebras preserve exact Borel subalgebras. Our next goal is to extend this statement to regular and homological exact Borel subalgebras. 

For any automorphism $\sigma$ of $A$, there is an endofunctor $\Phi_\sigma$ of $A\operatorname{-mod}$ defined by taking an $A$-module $V$ to the module having the same underlying vector space, but with the action defined by $a\cdot_\sigma v \coloneqq \sigma(a)\cdot v$. For ease of notation, we denote $\Phi_\sigma(V)$ by ${_\sigma}V$, indicating that the action of $A$ is twisted by the automorphism $\sigma$. If $B\subset A$ is a subalgebra, then so is $C\coloneqq \sigma(B)$. It is clear that $\Phi_\sigma$ restricts, in a natural way, to a functor $C\operatorname{-mod}\to B\operatorname{-mod}$. 

Let $b_1,\dots, b_n$ be some complete set of primitive orthogonal idempotents for $B\subset A$. Then, the set $\{b_1,\dots, b_n\}$ indexes the isomorphism classes of simple $B$-modules, and we denote by $L_B(i)$, where $1\leq i\leq n$, the simple module having as a basis the idempotent $b_i$ and the action defined by letting $b_i$ act as the identity and other basis elements as 0. Since $\sigma:B\to C$ is an isomorphism, $\sigma(b_1),\dots, \sigma(b_n)$ is a complete set of primitive orthogonal idempotents for $C$, and putting $\sigma(b_i)=c_i$, for $1\leq i\leq n$, we denote by $L_C(i)$ the simple $C$-module having as a basis the idempotent $c_i$ and the action defined by letting $c_i$ act as the identity and other basis elements as 0. With this notation, consider the module ${_\sigma} L_C(i)$ and fix the basis $c_i$. On this basis, the idempotent acts as $b_i\cdot_\sigma c_i\coloneqq \sigma(b_i)\cdot c_i=c_i\cdot c_i=c_i$. Thus, with this notation, we have that ${_\sigma}L_C(i)=L_B(i)$.

\begin{lemma}\label{lemma:inner auto preserves regular exact borel}
Let $A$ be an algebra and let $B\subset A$ be a subalgebra. Let $\sigma$ be an inner automorphism of $A$, given by $\sigma(x)=u^{-1}xu$ for some invertible element $u$ and $x\in A$. Put $C=\sigma(B)$.
\begin{enumerate}[(i)]
	\item There is an isomorphism of $A\operatorname{-}A\operatorname{-bimodules}$ $A\cong {_\sigma}A$ given by left multiplication with $u^{-1}$.
	\item For any $C$-module $M$, there is an isomorphism of left $A$-modules ${_\sigma}A\otimes_C M\to A\otimes_B {_\sigma}M.$
	\item For any $C$-module $M$ there is an isomorphism of left $A$-modules $A\otimes_C M\to A\otimes_B {_\sigma}M$.
\end{enumerate}
\end{lemma}
\begin{proof}
	\begin{enumerate}[(i)]
		\item That the map $x\mapsto u^{-1}(x)$ is a linear bijection on $A$ is clear, so we check that it is a homomorphism of bimodules. Indeed,
		$$u^{-1}(axb)=u^{-1}a u u^{-1} x b=\sigma(a) u^{-1}(x)b=a\cdot_\sigma u^{-1}(x) \cdot b.$$
		\item Define a map $\varphi: {_\sigma}A\otimes_C M \to A\otimes_B {_\sigma}M$ on generators by $a\otimes m \mapsto \sigma^{-1}(a)\otimes m$ and extend by linearity. Then, $\varphi$ is well-defined:
		\begin{align*}
			\varphi(ac\otimes_C m)&=\sigma^{-1}(ac)\otimes_B m=\sigma^{-1}(a)\sigma^{-1}(c)\otimes_B m=\sigma^{-1}(a)\otimes_B \sigma^{-1}(c)\cdot_\sigma m\\
			&=\sigma^{-1}(a)\otimes_B c\cdot m=\varphi(a\otimes_C c\cdot m),
		\end{align*}
	for any $c\in C$. Moreover, $\varphi$ is a homomorphism of left $A$-modules:
	\begin{align*}
\varphi(x\cdot_{\sigma} (a\otimes_C m))&=\varphi (\sigma(x)a\otimes_C m)=\sigma^{-1}(\sigma(x)a)\otimes_B m=x\sigma^{-1}(a)\otimes_B m=x\cdot( \sigma^{-1}(a)\cdot_B m)=x\cdot \varphi(a\otimes_C m)
	\end{align*}
for any $x\in A$. Define a map $\psi:A\otimes_B {_\sigma}M \to {_\sigma}A\otimes _C M$ on generators by $a\otimes m \mapsto \sigma(a)\otimes m$ and extend by linearity. Then,
$$\psi \varphi(a\otimes_C m)=\psi (\sigma^{-1}(a)\otimes_B m)=a\otimes_Cm\quad \text{and}\quad \varphi \psi(a\otimes_Bm)=\varphi(\sigma(a)\otimes_Cm)=a\otimes_B m,$$
so $\varphi$ and $\psi$ are mutually inverse bijections. We conclude that $\varphi$ is an isomorphism.
\item Combine (i) and (ii).
\qedhere
	\end{enumerate}
\end{proof}
\begin{theorem}\label{theorem:inner automorphisms preserve regular exact borel subalgebras}
Let $A$ be a basic quasi-hereditary algebra and let $B\subset A$ be a regular (respectively, homological) exact Borel subalgebra. Let $\sigma$ be an inner automorphism of $A$. Then, $\sigma(B)$ is again a regular (respectively, homological) exact Borel subalgebra of $A$.
\end{theorem}
\begin{proof}
Due to Theorem \ref{theorem: inner auto preserves borel}, we know that $C$ is an exact Borel subalgebra. We interpret elements of the spaces $\operatorname{Ext}_B^k(L_B(i),L_B(j))$ and $\operatorname{Ext}_C^k(L_C(i),L_C(j))$ as exact sequences of length $k+2$, modulo equivalence. Since $A \otimes_B \blank$ and $A\otimes_C\blank$ are exact functors, applying them to exact sequences in $\operatorname{Ext}_B^k(L_B(i),L_B(j))$ and $\operatorname{Ext}_C^k(L_C(i),L_C(j))$ produces exact sequences in the spaces
 $$\operatorname{Ext}_A^k(A\otimes_B L_B(i), A\otimes_B L_B(j))\quad \textrm{and}\quad \operatorname{Ext}_A^k(A\otimes_C L_C(i), A\otimes_C L_C(j)).$$
 
 Functoriality ensures that these maps are well-defined. Denote them by $F_B$ and $F_C$, respectively. We denote by ${_\sigma}\blank$ the map $\operatorname{Ext}_C^k(L_C(i),L_C(j))\to \operatorname{Ext}_B^k(L_B(i),L_B(j))$ defined by taking an exact sequence
 $$L_C(j) \to M_k \to \dots \to M_1 \to L_C(i) \in \operatorname{Ext}_C^k(L_C(i),L_C(j))$$ to the sequence
 \begin{align*}
 	_\sigma L_C(j) \to  _\sigma M_k \to   &\dots \to _\sigma M_1 \to _\sigma L_C(i) \\
 	&=\\
 	L_B(j) \to _\sigma M_k \to &\dots \to _\sigma M_1 \to L_B(i).
 \end{align*}
  Consider the following diagram.
$$\xymatrix{
\operatorname{Ext}_B^k(L_B(i), L_B(j)) \ar[d]_-{F_B} & & \operatorname{Ext}_C^k(L_C(i), L_C(j)) \ar[d]^-{F_C} \ar[ll]_-{{_\sigma}\blank}\\
\operatorname{Ext}_A^k(A\otimes_B L_B(i),A\otimes_B L_B(j)) \ar[rd]_-{\sim} & & \operatorname{Ext}_A^k(A\otimes_C L_C(i), A\otimes_C L_C(j)) \ar@{-->}[ll] \ar[ld]^-{\sim} \\
&\operatorname{Ext}_A^k(\Delta(i),\Delta(j))
}$$
Let $L_C(j) \to M_k \to \dots \to M_1 \to L_C(i)$ be an exact sequence, interpreted as an element of the space $\operatorname{Ext}_C^k(L_C(i), L_C(j))$. Applying first ${_\sigma}\blank$ and then $F_B$, we obtain
$$A\otimes_B L_B(j)\to A\otimes_B{_\sigma} M_k\to \dots \to A\otimes_B{_\sigma}M_1 \to A\otimes_B L_B(i).$$
If we instead apply $F_C$ to our original sequence $L_C(j)\to M_k \to \dots \to M_1 \to L_C(i)$, we get
$$A\otimes_C L_C(j) \to A\otimes_C M_k \to \dots \to A\otimes_C M_1 \to A\otimes_C L_C(i).$$
Applying Lemma \ref{lemma:inner auto preserves regular exact borel}, we know there is an isomorphism of left $A$-modules, given by the composite
$$A\otimes_C M \to {_\sigma}A\otimes_C M \to A\otimes_B {_\sigma}M.$$
This means that we may choose the dashed arrow in the diagram to be the map
\begin{align*}
	A\otimes_C L_C(j)\to A\otimes_C M_k \to & \dots\to A\otimes_C M_1 \to A\otimes_C L_C(i) \\
	&\mapsto \\
	A\otimes_B L_B(j) \to A\otimes_B {_\sigma} M_k \to &\dots \to A\otimes_B {_\sigma}M_1 \to A\otimes_B L_B(i),
\end{align*}
which is an isomorphism and makes the top square commute, by construction. It follows that $F_C$ is an isomorphism. For the second statement, the same argument as above ensures that the maps $\operatorname{Ext}^k_C(L_C(i),L_C(j))\to \operatorname{Ext}^k_A(A\otimes L_C(i),A\otimes L_C(j))$ are isomorphisms for $k\geq 2$. For $k=1$, the above diagram implies that we can solve for $F_C$ to see that $F_C$ is the composition of three epimorphisms, hence an epimorphism.
\end{proof}
\begin{corollary}\label{corollary:theorem:isomorphisms preserve regular exact borel subalgebras}
Let $A=KQ/I$ be a basic quasi-hereditary  algebra containing a regular exact Borel subalgebra. Then, all exact Borel subalgebras of $A$ are regular. In particular, $A$ admits a regular exact Borel subalgebra containing the idempotents $e_1,\dots, e_n$.
\end{corollary}
\begin{proof}
By assumption, there exists a regular exact Borel subalgebra $B$ of $A$. Let $C$ be an exact Borel subalgebra of $A$. By Theorem \ref{theorem: uniqueness of borel}, there exists an inner automorphism $\sigma$ of $A$ such that $\sigma(B)=C$, and since inner automorphisms preserve regular exact Borel subalgebras, by Theorem \ref{theorem:inner automorphisms preserve regular exact borel subalgebras}, $C$ is a regular exact Borel subalgebra. In the discussion following Theorem \ref{theorem: uniqueness of borel}, we saw that $A$ admits an exact Borel subalgebra containing the idempotents $e_1,\dots, e_n$. By the above argument, it is regular.
\end{proof}
\begin{proposition}\label{proposition:what paths are in regular exact borel}
	Let $B\subset A_n$ be a regular exact Borel subalgebra containing the idempotents $e_1,\dots, e_n$ and suppose that $\Delta(i)=M(i, s_i)$.
	\begin{enumerate}[(i)]
		\item If $i\leq s_i<n$, then $\alpha_j\dots \alpha_i \notin B$, for any $i< j<s_i$, and $\alpha_{s_i}\dots \alpha_i \in B$.
		\item If $s_i=n$, that is, if $\Delta(i)$ is projective, then $\alpha_j\dots \alpha_i \notin B$, for any $i\leq j<n$.

	\end{enumerate}
\end{proposition}
\begin{proof}
	\begin{enumerate}[(i)]
		\item Assume towards a contradiction that $\alpha_j \dots \alpha_i \in B$ for some $i<j<s_i$. Let $p\in A_n$ be a path and consider the generator $p\otimes e_i$ of $A_n\otimes_B L(i)$. If $s(p)\neq i$, then
		$$e_{j+1}\cdot (p\otimes e_i)= e_{j+1}\cdot (p\otimes e_{s(p)}e_i)=0,$$
		since $e_{s(p)}\in B$. Similarly, if $t(p)\neq j+1$, then
		$$e_{j+1}\cdot (p\otimes e_i)=e_{j+1}e_{t(p)} p\otimes e_i=0.$$
		If $s(p)=i$ and $t(p)=j+1$, then $p=\alpha_j\dots \alpha_i$, which we assume is contained in $B$. We get
		$$\alpha_j\dots \alpha_i \otimes e_i=e_j\otimes \alpha_j \dots \alpha_i e_i=0,$$
		so that $e_{j+1} \cdot (A_n\otimes_B L(i))=0.$ However, $e_{j+1}\cdot \Delta(i)\neq 0$, which is our desired contradiction.
		
		For the second statement, assume towards a contradiction that $\alpha_{s_i}\dots \alpha_i\notin B$. Then, $\alpha_{s_i}\dots \alpha_i \otimes e_i$ is a non-zero element of $A_n \otimes_B L(i)$. To see this, note that $\alpha_j\dots \alpha_i \notin B$ for any $i<j<s_i$, by the first part. It follows that also $e_{s_i+1} \cdot (\alpha_{s_i}\dots \alpha_i \otimes e_i)\neq 0$. Now, we see that $e_{s_i+1}\cdot \Delta(i)=0$, while $e_{s_i+1}\cdot A_n\otimes_B L(i)$, leading to a similar contradiction as above.

\item Similar to the first part of (i). We remark that this is a separate case only because the first statement of part (ii) does not make sense when $\Delta(i)$ is projective. \qedhere
	\end{enumerate}
\end{proof}
\begin{proposition}\label{proposition:generating set for regular exact borel}
	Let $B\subset A_n$ be a regular exact Borel subalgebra containing the idempotents $e_1,\dots, e_n$. Let $S_B$ be the set of paths which are contained in $B$ by Proposition \ref{proposition:what paths are in regular exact borel} together with the idempotents $e_1,\dots, e_n$. Then, $S_B$ is a minimal generating set for $B$.
\end{proposition}
\begin{proof}
It is well known that any basic and connected finite-dimensional algebra over an algebraically closed field $K$ is isomorphic to the quotient of the path algebra of some quiver by some admissible ideal. We recall the construction of this isomorphism. For details, we refer to \cite[Theorem~3.7]{ASS}.

Let $\Phi:KQ_B\to B$ be the isomorphism mentioned above. Then, $\Phi$ is defined on the trivial paths in $Q_B$ by $\Phi(e_i)=e_i$, for $1\leq i\leq n$. Note that this makes sense, as we may talk about the idempotents $e_i$ as elements of $KQ_B$ and of $B$.

Consider the set of arrows $i\to j$ in $Q_B$. As discussed prior to Proposition \ref{proposition:quiver of regular exact borel}, the cardinality of this set equals $\dim \operatorname{Ext}_{A_n}^1(\Delta(i),\Delta(j))$, which may be 0 or 1, according to Theorem \ref{theorem:ext-algebra of A_n}. Therefore, this set is either empty or contains a single arrow, which we denote by $x_i^j$. The arrow $x_i^j$ should be mapped to an element $y\in \operatorname{rad}B$, so that the residue class $y+\operatorname{rad}^2B$ forms a basis of $e_j \faktor{\operatorname{rad}B}{\operatorname{rad}^2B}e_i$. This is achieved by defining $\Phi(x_i^j)=\alpha_{j-1}\dots \alpha_i$. Since we know that $\Phi$ is an isomorphism, we know that the idempotents $e_i$, together with paths of the form $\Phi(x_i^j)$, constitute a minimal generating set for $B$.

Recall that the arrow $x_i^j\in Q_B$ exists precisely when $\dim \operatorname{Ext}_{A_n}^1(\Delta(i),\Delta(j))=1$. Consider the following picture. The dashed edge between $b$ and $c$ signifies that $c$ is a vertex in the right subtree of $b$ (not necessarily the vertex $r(b)$), such that there exists a non-zero homomorphism from $\Delta(c)$ to $\Delta(b)$.

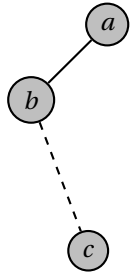
$$\begin{tikzpicture}
	\node(a)[shape=circle,draw, fill=lightgray, thick] at (0,0) {$a$};
	\node(b) [shape=circle,draw, fill=lightgray, thick] at (-1,-1) {$b$};
	\node(c) [shape=circle,draw, fill=lightgray, thick] at (-0.25, -3) {$c$};
	\draw[thick] (a) to (b);
	\draw[thick,dashed] (b) to (c);
\end{tikzpicture}$$
We know, using Theorem \ref{theorem:ext-algebra of A_n}, that possible non-zero extensions between standard modules appear as either $\operatorname{Ext}_{A_n}^1(\Delta(b), \Delta(a))$ or as $\operatorname{Ext}_{A_n}^1(\Delta(c),\Delta(a))$, where $a$, $b$ and $c$ are configured as above. Since $b$ labels $\ell(a)$, the vertex $a$ is the first vertex visited immediately after visiting the entire right subtree of $a$. Therefore, $a=s_b+1$. Since $c$ is in the right subtree of $b$, we have $s_b=s_c$, and therefore $a=s_c+1$. Plugging this into the above, we see that the generators of $B$ are of the form
$$\Phi(x_i^{s_i+1})=\alpha_{s_i}\dots \alpha_i,$$
whence it follows that $S_B$ is a minimal generating set. \qedhere
\end{proof}
\begin{corollary}\label{corolllary:proposition:generating set for regular exact borel}
Let $B$ be the exact Borel subalgebra $B\subset A_n$ containing the idempotents $e_1,\dots, e_n$. Then, the exact Borel subalgebras are precisely those given as $C=\langle u^{-1}S_Bu\rangle$, where $u\in A_n$ is some invertible element.
\end{corollary}
\begin{example}
	We consider the binary search tree
	$$
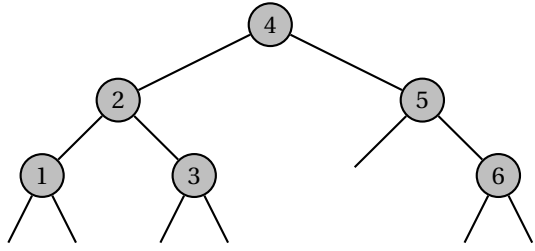
\begin{tikzpicture}
		\node(a) [shape=circle, draw, thick, fill=lightgray] at (0,0) {4};
		\node(b) [shape=circle, draw, thick, fill=lightgray] at (-2,-1) {2};
		\node(c) [shape=circle, draw, thick, fill=lightgray] at (2,-1) {5};
		\node(d) [shape=circle, draw, thick, fill=lightgray] at (-1,-2) {3};
		\node(e) [shape=circle, draw, thick, fill=lightgray] at (-3,-2) {1};
		\node(f)  at (1,-2) {};
		\node(g) [shape=circle, draw, thick, fill=lightgray] at (3,-2) {6};
		\node(h) at (-3.5,-3) {};
		\node(i)  at (-2.5,-3) {};
		\node(j)  at (-1.5,-3) {};
		\node(k)  at (-0.5,-3) {};
		\node(n)  at (2.5,-3) {};
		\node(o) at (3.5,-3) {};
		\draw[thick] (a) to (b);
		\draw[thick] (b) to (e);
		\draw[thick] (e) to (h);
		\draw[thick] (e) to (i);
		\draw[thick] (b) to (d);
		\draw[thick] (d) to (j);
		\draw[thick] (d) to (k);
		\draw[thick] (a) to (c);
		\draw[thick] (c) to (f);
		\draw[thick] (c) to (g);
		\draw[thick] (g) to (n);
		\draw[thick] (g) to (o);
	\end{tikzpicture}$$
We saw, in the example following Theorem \ref{theorem:ext-algebra of A_n}, that the Ext-algebra of standard modules over $A_n$ is given by the quiver
	$$
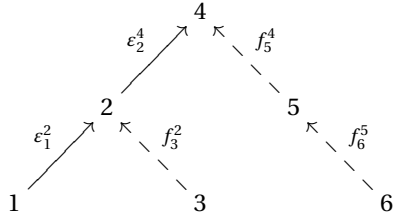
\xymatrix{
		& & 4\\
		& 2\ar[ru]^-{\varepsilon_2^4} & & 5\ar@{-->}[lu]_-{f_5^4}\\ 
		1 \ar[ru]^-{\varepsilon_1^2}	& & 3 \ar@{-->}[lu]_-{f_3^2} & &6\ar@{-->}[lu]_-{f_6^5}
	}$$
	modulo the relations $I=\langle \varepsilon_2^4\varepsilon_1^2\rangle$. Here, there are three non-zero spaces of extensions, namely
	$$\dim\operatorname{Ext}_{A_n}^1(\Delta(1),\Delta(2))=\dim\operatorname{Ext}_{A_n}^1(\Delta(2),\Delta(4))=\dim\operatorname{Ext}_{A_n}^1(\Delta(3),\Delta(4))=1.$$
	Letting $B\subset A_n$ be the regular exact Borel subalgebra containing $e_1,\dots, e_n$, we see that the quiver of $B$ is 
	$$
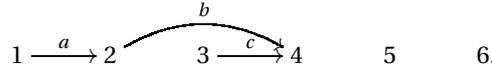
\xymatrix{
	1 \ar[r]^-a & 2 \ar@/^1pc/[rr]^-b & 3 \ar[r]^-c & 4 & 5 & 6.
	}$$
The arrows $a$, $b$ and $c$ correspond to the paths $\alpha_1, \alpha_3\alpha_2$ and $\alpha_3$, respectively. A generating set $S$ for $B$ is then
$$S=\{e_1, e_2, e_3, e_4, e_5, e_6, \alpha_1, \alpha_3, \alpha_3\alpha_2\}.$$
	\end{example}
\subsection{$A_\infty$-structure on $\operatorname{Ext}_{A_n}^\ast(\Delta,\Delta)$}
The notion of an $A_\infty$-algebra is meant to capture the idea of an algebra which is not strictly associative, but associative only up to a system of ``higher homotopies''. There is also the natural ``multi-object'' version, called an $A_\infty$-cateogory, which we define below. In this section we are particularly interested in $\operatorname{Ext}_{A_n}^\ast(\Delta,\Delta)$, which always carries the structure of an $A_\infty$-algebra. However, we will view $\operatorname{Ext}_{A_n}^\ast(\Delta,\Delta)$ as an $A_\infty$-category with $n$ objects, given by the standard modules. This will be useful since it enables us to enforce compatibility of the higher homotopies with the natural idempotents of $\operatorname{Ext}_{A_n}^\ast(\Delta,\Delta)$: the identity homomorphisms $1_{\Delta(i)}$. 
\begin{definition}\cite{Keller2}
	Let $K$ be a field. An $A_\infty$-category $\mathcal{A}$ consists of the following:
	\begin{enumerate}[(i)]
		\item a class of objects $\operatorname{Ob}(\mathcal{A})$;
		\item for each pair of objects $A$ and $B$, a $\mathbb{Z}$-graded vector space $\operatorname{Hom}_{\mathcal{A}}(A,B)$;
		\item for all $n\geq 1$ and objects $A_0,\dots, A_n$, a homogeneous linear map
		$$m_n: \operatorname{Hom}_{\mathcal{A}}(A_{n-1},A_n) \otimes \operatorname{Hom}_{\mathcal{A}}(A_{n-2},A_{n-1}) \otimes \dots \otimes \operatorname{Hom}_{\mathcal{A}}(A_0, A_1) \to \operatorname{Hom}_{\mathcal{A}}(A_0, A_n)$$
		of degree $2-n$ such that
	$$\sum_{r+s+t=n}(-1)^{r+st}m_{r+1+t}(1^{\otimes r} \otimes m_s \otimes 1^{\otimes t})=0.$$
	\end{enumerate}
		\end{definition}
	Viewing $\operatorname{Ext}_{A_n}^\ast(\Delta,\Delta)$ as an $A_\infty$-category as outlined above amounts to the following: if we have extensions $\varepsilon_i$ arranged as
	$$\xymatrix{
\Delta(i_1) \ar[r]^-{\varepsilon_1} & \Delta(i_2) \ar[r]^-{\varepsilon_2} & \dots \ar[r] & \Delta(i_\ell) \ar[r]^-{\varepsilon_{\ell}} & \Delta(i_ {\ell+1})	
}$$
then $m_\ell(\varepsilon_\ell,\dots, \varepsilon_1)$ is an extension from $\Delta(i_1)$ to $\Delta(i_{\ell+1})$.
Writing out the $A_\infty$-relations for the first few $n$, we get:
	\begin{enumerate}
		\item[$n=1$:] $m_1m_1=0$, meaning that $A$ is a complex.
		\item[$n=2$:] $m_1m_2=m_2(m_1\otimes 1 + 1\otimes m_1)$, meaning that $m_1$ is a derivation with respect to $m_2$.
	\item[$n=3$:] $m_2(1\otimes m_2-m_2\otimes 1)=m_1m_3+m_3(m_1\otimes1\otimes1+1\otimes m_1\otimes 1+ 1\otimes 1\otimes m_1)$, meaning that $m_2$ is associative up to a homotopy given by $m_3$.
	\end{enumerate}
When evaluating at actual elements, even more signs appear according to the Koszul sign rule:
$$(f\otimes g)(x\otimes y)=(-1)^{|g||x|}f(x)\otimes g(y).$$
Here, $f$ and $g$ are homogeneous maps and $x$ and $y$ are homogeneous elements. The vertical bars denote degree.
A graded category $A$ such that any $A_\infty$-structure on it satisfies $m_n=0$ for $n\geq 3$ is called \emph{intrinsically formal}.

\begin{proposition}\label{proposition:a-infty on A_n is trivial} $\operatorname{Ext}_{A_n}^\ast(\Delta,\Delta)$ is intrinsically formal.
\end{proposition}
\begin{proof}
Let $\varphi_1,\dots, \varphi_\ell$ be homogeneous extensions in $\operatorname{Ext}_{A_n}^\ast(\Delta,\Delta)$. Put $i=\sum_{j=1}^\ell \deg \varphi_j$. Since $m_\ell$ is of degree $2-\ell$, the extension $m_\ell(\varphi_\ell,\dots, \varphi_1)$ is of degree $2-\ell+i$. Because $A_n$ is hereditary, the extension $m_\ell(\varphi_\ell,\dots, \varphi_1)$ is of degree 0 or 1 if it is non-zero. This implies that there are the following cases:
$$2-\ell+i=\begin{cases}
	0, & \textrm{if }  i=\ell-2 \\
	1, & \textrm{if } i=\ell-1.
\end{cases}$$
Assume that $i=\ell-1$ and that $m_\ell(\varphi_\ell,\dots,\varphi_1)$ produces an extension in $\operatorname{Ext}_{A_n}^1(\Delta(c),\Delta(a))$. We know that if such an extension is non-zero, the vertices $a$ and $c$ must be configured in the following way:
here the dashed edge represents a path which is the concatenation of edges between a vertex $v$ and $r(v)$.

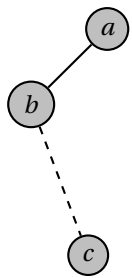
$$\begin{tikzpicture}
	\node(a)[shape=circle,draw, fill=lightgray, thick] at (0,0) {$a$};
	\node(b) [shape=circle,draw, fill=lightgray, thick] at (-1,-1) {$b$};
	\node(c) [shape=circle,draw, fill=lightgray, thick] at (-0.25, -3) {$c$};
	\draw[thick] (a) to (b);
	\draw[thick,dashed] (b) to (c);
\end{tikzpicture}$$
Since $i=\ell-1$, there is precisely one argument $\varphi_j$ of $m_\ell(\varphi_\ell,\dots,\varphi_1)$ which is a homomorphism rather than a proper extension. Consider the following picture, where the dashed edge represents the homomorphism $\varphi_j:\Delta(x)\to \Delta(y)$ and the wavy arrows from $c$ to $x$ and from $y$ to $a$ represent the compositions $\varphi_{j-1},\dots, \varphi_1$ and $\varphi_\ell,\dots,\varphi_{j+1}$, respectively.
$$
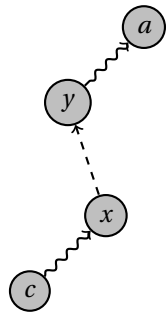
\begin{tikzpicture}
	\node(a)[shape=circle,draw, fill=lightgray, thick] at (0,0) {$a$};
	\node(b) [shape=circle,draw, fill=lightgray, thick] at (-1,-1) {$y$};
	\node(c) [shape=circle,draw, fill=lightgray, thick] at (-0.5, -2.5) {$x$};
	\node(d) [shape=circle,draw, fill=lightgray, thick] at (-1.5, -3.5) {$c$};
	\draw[thick, <-,snake=snake, segment amplitude=.4mm, segment length=2mm] (a) to (b);
	\draw[thick,dashed, <-] (b) to (c);
	\draw[thick, ->,snake=snake, segment amplitude=.4mm, segment length=2mm] (d) to (c);
\end{tikzpicture}$$ 
 Note that either of the wavy edges may represent a path of length $>1$, rather than actual edges. Moreover, depending on $j$, the edge between $c$ and $x$ or the edge between $y$ and $a$ may be ``empty''. We see that it is impossible to have the allowed configuration from the previous picture.
 
 If $i=\ell-2$, the outcome of $m_\ell(\varphi_\ell,\dots \varphi_1)$ is an element of $\operatorname{Hom}_{A_n}(\Delta(c),\Delta(a))$. Since $\ell\geq 3$, there is at least one argument $\varphi_j$ of $m_\ell(\varphi_\ell,\dots,\varphi_1)$ which is a proper extension. According to the allowed configuration above, this means that $c$ is in the left subtree of $a$. Using Theorem \ref{theorem:ext-algebra of A_n}, we get that $\operatorname{Hom}_{A_n}(\Delta(c),\Delta(a))=0$.
\end{proof}
\section{$\operatorname{Ext}$-algebras of standard modules over deconcatenations}
In the previous section, we studied the path algebra of $\mathbb{A}_n$ with linear orientation. To extend this study to arbitrary orientations, we apply the method of ``deconcatenation'' of a quiver $Q$ at a sink or a source. In \cite{FKR}, the authors use this method to study the different quasi-hereditary structures of the path algebra $KQ$ in terms of the quasi-hereditary structures on path algebras of certain subquivers of $Q$. In the present section, we extend this study to the $\operatorname{Ext}$-algebra of standard modules. 
\begin{definition}\cite[Definition~3.1]{FKR}
	Let $Q$ be a finite connected quiver and let $v$ be a vertex of $Q$ which is a sink or a source. A \emph{deconcatenation} of $Q$ at the vertex $v$ is a union $Q^1\sqcup \dots \sqcup Q^\ell$ of full subquivers $Q^i$ of $Q$ such that the following hold.
	\begin{enumerate}[(i)]
		\item Each $Q^i$ is a connected full subquiver of $Q$ with $v\in Q^i_0$.
		\item For all $i,j\in \{1,\dots, \ell\}$, such that $i\neq j$, there holds $Q_0=\left(Q_0^1\backslash\{v\}\right)\sqcup \dots \sqcup \left(Q_0^\ell\backslash\{v\}\right) \sqcup \{v\}$ and $Q_0^i\cap Q_0^j=\{v\}$.
		\item For $x\in Q_0^i\backslash\{v\}$ and $y\in Q_0^j\backslash\{v\}$, where $i,j\in\{1,\dots, \ell\}$ are such that $i\neq j$, there are no arrows between $x$ and $y$ in $Q$.
	\end{enumerate}
\end{definition}
\begin{example}
	Consider the quiver $Q=\xymatrixcolsep{0.5cm}\xymatrix{1 & 2 \ar[l]& 3\ar[r] \ar[l] & 4\ar[r] & 5}$. We see that $Q$ has a deconcatenation at the source $3$:
	$$(\xymatrixcolsep{0.5cm}\xymatrix{1 & 2 \ar[l] & 3\ar[l]}) \sqcup (\xymatrixcolsep{0.5cm}\xymatrix{3 \ar[r] &4 \ar[r] & 5}).$$
\end{example}

We follow the notation of \cite{FKR} and put
\begin{align*}
	A^\ell = \faktor{A}{\langle e_u \ |\  u\in Q_0 \backslash Q_0^\ell \rangle }.
\end{align*}
Then, $A$ surjects onto $A^\ell$, and  consequently, there is a fully faithful and exact functor $F^\ell: A^\ell\operatorname{-mod}\to A\operatorname{-mod}$. 
Regarding $A^\ell\operatorname{-mod}$ as a full subcategory of $A\operatorname{-mod}$ via the functor $F^\ell$, we remark that an $A$-module $M$ is an $A^\ell$-module if and only if $e_uM=0$ for any $u\in Q_0\backslash Q_0^{\ell}$.

\begin{lemma}\label{lemma:morphisms lift between from part of deconcatenation}
	Let $M$ be an $A^\ell$-module, let $N$ be an $A$-module and consider the quotient $\overline{N}=\faktor{N}{\sum e_u\cdot N}$, where the sum ranges over all $u\in Q_0^{\overline{\ell}}\backslash\{v\}$. Then, $\overline{N}$ is an $A^\ell$-module and there are isomorphisms of vector spaces
	$$\operatorname{Hom}_A(M, \overline{N})\cong\operatorname{Hom}_{A^\ell}(M, \overline{N})\cong \operatorname{Hom}_A(M, N).$$
\end{lemma}
\begin{proof}The first isomorphism is immediate from the fact that $F^\ell$ is fully faithful. To see the second isomorphism, put $e^\ell=\sum e_u$ where the sum is as in the statement of the Lemma. Then, there is an isomorphism of $A$-modules
$$\overline{N}\cong\operatorname{Hom}_A(A/\langle e^\ell \rangle, N).$$
Next, note that by tensor-$\operatorname{Hom}$ adjunction, there is an adjoint pair of functors $(F^\ell, \operatorname{Hom}_A(A/\langle e^\ell\rangle, \blank))$. Then,
$$\operatorname{Hom}_A(M, N)\cong \operatorname{Hom}_A(F^\ell(M), N)\cong \operatorname{Hom}_{A^\ell}(M, \operatorname{Hom}_A(A/\langle e^\ell \rangle, N))\cong \operatorname{Hom}_{A^\ell}(M, \overline{N}).$$
\end{proof}
\begin{lemma}\cite[Lemma~3.3]{FKR}\label{lemma:elem. properties of deconcatenation}
	Let $Q^1\sqcup Q^2$ be a deconcatenation of $Q$ at a sink or source $v$ and let $\ell=1,2$.
	\begin{enumerate}[(i)]
		\item For all $i\in Q_0^\ell$, we have $L(i)\cong L^\ell(i)$.
		\item For all $i\in Q_0^\ell \backslash \{v\}$, we have $P(i)\cong P^\ell(i)$ and $I(i)\cong I^\ell(i)$.
		\item If $v$ is a sink, we have $L(v)\cong P(v) \cong P^\ell(v)$, for $\ell=1,2$.
		\item If $v$ is a source, we have $L(v) \cong I(v)\cong I^\ell(v)$, for $\ell=1,2$. 
		\item For any non-zero $A$-module $M$, if both $\operatorname{top}M$ and $\operatorname{soc}M$ are simple, then, we have either $M\in A^1\operatorname{-mod}$ or $M\in A^2 \operatorname{-mod}$. 
		\item Let $M$ be an $A^\ell$-module and let $i\in Q_0$ be a vertex. If $[M :L(i)]\neq 0$, then $i\in Q_0^\ell$.
	\end{enumerate}
\end{lemma}
Given a partial order $\trianglelefteq$ on $Q_0$ and a deconcatenation $Q_1\sqcup Q_2$, of $Q$ at a sink or source $v$, we may form the restriction $\trianglelefteq \mid_{Q_0^\ell}$ to obtain a partial order $\trianglelefteq^\ell=\trianglelefteq \mid_{Q_0^\ell}$ on $Q_0^\ell$.

\begin{lemma}\cite[Lemma~3.4]{FKR}\label{lemma:qh-structure on parts from qh-structure of whole}
	Let $Q_1\sqcup  Q_2$ be a deconcatenation of $Q$ at a sink or source $v$ and let $\trianglelefteq$ be a partial order on $Q_0$. Denote by $\Delta$ and $\nabla$ the sets of standard and costandard  $A$-modules associated to $\trianglelefteq$, respectively. Denote by $\Delta^\ell$ and $\nabla^\ell$ the sets of standard and costandard $A^\ell$-modules, respectively.
	
	\begin{enumerate}[(i)]
		\item For any $i\in Q_0^\ell\backslash\{v\}$, there are isomorphisms of $A$-modules $\Delta(i)\cong \Delta^\ell(i)$ and $\nabla(i)\cong \nabla^\ell(i)$, for $\ell=1,2$.
		\item If $v$ is a sink, there are isomorphisms of $A$-modules $L(v)\cong \Delta(v)\cong \Delta^\ell(v)$.
		\item If $v$ is a source, there are isomorphisms of $A$-modules $L(v)\cong \nabla(v) \cong \nabla^\ell(i)$.
		\item If $A$ is quasi-hereditary with respect to $\trianglelefteq$, then $A^\ell$ is quasi-hereditary with respect to $\trianglelefteq^\ell$, for $\ell=1,2$.
	\end{enumerate}
\end{lemma}
%\begin{proof}
	%	content...
	%\end{proof}
Let $Q=Q^1\sqcup Q^2$ be a deconcatenation of $Q$ at a source or sink $v$ and let $\trianglelefteq^\ell$ be a partial order on $Q_0^\ell$ for $\ell=1,2$. Put $\overline{1}=2$ and $\overline{2}=1$. Define a partial order $\trianglelefteq=\trianglelefteq(\trianglelefteq^1,\trianglelefteq^2)$ on $Q_0$ as follows. For $i,j\in Q_0$, we say that $i\triangleleft j$ if one of the following hold.
\begin{enumerate}[(i)]
	\item We have $i,j\in Q_0^\ell$ and $i\triangleleft^\ell j$, for some $\ell$.
	\item We have $i\in Q_0^\ell$, $j\in Q_0^{\overline{\ell}}$, $i\triangleleft^\ell v$ and $v\triangleleft^{\overline{\ell}} j$.
\end{enumerate}
\begin{lemma}\cite[Lemma~3.5]{FKR}\label{lemma:qh-structure on whole from qh-structure on deconcatenation}
	Let $Q^1\sqcup Q^2$ be a deconcatenation of $Q$ at a sink or source $v$. Let $\trianglelefteq^\ell$ be a partial order on $Q_0^\ell$ and denote by $\Delta^\ell$ and $\nabla^\ell$ the sets of standard and costandard $A^\ell$-modules, associated to $\trianglelefteq^\ell$, for $\ell=1,2$, respectively. Denote by $\Delta$ and $\nabla$ the sets of standard and costandard $A$-modules, respectively, associated to $\trianglelefteq=\trianglelefteq(\trianglelefteq^1,\trianglelefteq^2)$. Then, we have the following.
	\begin{enumerate}[(i)]
		\item For any $i\in Q_0^\ell \backslash\{v\}$, there are isomorphisms of $A$-modules $\Delta(i)\cong \Delta^\ell(i)$ and $\nabla(i)\cong \nabla^\ell(i)$.
		\item If $v$ is a sink, there are isomorphisms of $A$-modules $L(v)\cong \Delta(v)\cong \Delta^\ell(v)$.
		\item If $v$ is a source, there are isomorphisms of $A$-modules $L(v)\cong \nabla(v)\cong \nabla^\ell(i)$.
		\item If $\trianglelefteq^\ell$ defines a quasi-hereditary structure on $A^\ell$ for $\ell=1$ and $\ell=2$, then $\trianglelefteq=\trianglelefteq(\trianglelefteq^1,\trianglelefteq^2)$ defines a quasi-hereditary structure on $A$.
	\end{enumerate}
\end{lemma}
%\begin{example}
	%	We consider again the concatenation $(\xymatrixcolsep{0.5cm}\xymatrix{1 & 2 \ar[l] & 3\ar[l]}) \sqcup (\xymatrixcolsep{0.5cm}\xymatrix{3 \ar[r] &4 \ar[r] & 5})$
	%		of the quiver  $Q=\xymatrixcolsep{0.5cm}\xymatrix{1 & 2 \ar[l]& 3\ar[r] \ar[l] & 4\ar[r] & 5}$ and draw the Loewy diagrams of the structure modules for $A^1$, $A^2$ and $A$.
	%		
	%	\begin{align*}
		%		P(1)&:\vcenter{\xymatrixrowsep{0.5cm}\xymatrix{1}},\quad P(2):\vcenter{\xymatrixrowsep{0.5cm}\xymatrix{2\ar[d]\\1}},\quad P(3):\vcenter{\xymatrixcolsep{0.5cm}\xymatrixrowsep{0.5cm}\xymatrix{&3 \ar[ld]\ar[rd]\\2\ar[d] & & 4\ar[d]\\1 & & 5}},\quad P(4):\vcenter{\xymatrixrowsep{0.5cm}\xymatrix{4 \ar[d]\\5}},\quad P(5):\vcenter{\xymatrixrowsep{0.5cm}\xymatrix{5}}, \\
		%		\Delta(1)&\cong P(1),\quad \Delta(2)\cong P(2),\quad \Delta(3):\vcenter{\xymatrixrowsep{0.5cm}\xymatrix{3\ar[d]\\2\ar[d]\\1}},\quad \Delta(4)\cong P(4),\quad \Delta(5)\cong P(5)\\
		%		I(1)&\cong \Delta(3),\quad I(2): \vcenter{\xymatrixrowsep{0.5cm}\xymatrix{3\ar[d] \\ 2}}, \quad I(3):3,\quad I(4): \vcenter{\xymatrixrowsep{0.5cm}\xymatrix{3\ar[d] \\ 4}}, \quad I(5): \vcenter{\xymatrixrowsep{0.5cm}\xymatrix{3\ar[d]\\4\ar[d]\\5}}
		%	\end{align*}
	%\end{example}
%\begin{proof}
	%	content...
	%\end{proof}
We have the following observation about homomorphism spaces between indecomposable projective modules.
\begin{lemma}\label{lemma:no homs between projectives in different parts of deconcatenation}
	Let $Q^1\sqcup Q^2$ be a deconcatenation of $Q$ at the sink or source $v$. Then, for $i\in Q_0^1\backslash\{v\}$ and $j\in Q_0^2\backslash\{v\}$, there holds $\operatorname{Hom}_A(P(i),P(j))=\operatorname{Hom}_A(P(j),P(i))=0$.
\end{lemma}
\begin{proposition}\label{proposition:no ext between different parts of concatenation}
	Let $Q^1\sqcup Q^2$ be a deconcatenation of $Q$ at the sink or source $v$. Let $i\in Q^1_0\backslash\{v\}$ and $j\in Q^2_0\backslash\{v\}$. 
	\begin{enumerate}[(i)]
		\item If $v$ is a source, there holds 
		$$\operatorname{Ext}_A^k(\Delta(i),\Delta(j))=\operatorname{Ext}_A^k(\Delta(j),\Delta(i))=0,$$
		for all $k\geq 0.$
		\item If $v$ is a sink and maximal or minimal with respect to $\trianglelefteq^e$, there holds 
		$$\operatorname{Ext}_A^k(\Delta(i),\Delta(j))=\operatorname{Ext}_A^k(\Delta(j),\Delta(i))=0,$$
		for all $k\geq 0.$
	\end{enumerate}
\end{proposition}
\begin{proof}
	\begin{enumerate}[(i)]
		\item Let $v$ be a source and let $P^\bullet\to \Delta(i)$ be a minimal projective resolution. We claim that each term $P^k$ of the resolution $P^\bullet$ decomposes into a direct sum of indecomposable projective modules as
		$$P^k=\bigoplus_{\ell=1}^n P(\ell)^{m_{\ell,k}},$$
		where, if $m_{\ell,k}>0$, then $\ell \in Q_0^1\backslash\{v\}$. Since $v$ is a source, we must have $[P(i):L(v)]=0$. Consequently, if $\pi_i: P(i)\to \Delta(i)$ is the natural projection, we must also have $[\ker \pi_i: L(v)]=0$. It follows that the projective cover of $\ker \pi_i$ has no composition factor $L(v)$. Proceeding by induction, the claim holds. Applying $\operatorname{Hom}_A(\blank,\Delta(j))$ to $P^\bullet$, we consider the space $\operatorname{Hom}_A(P^k,\Delta(j))$. This space decomposes into a direct sum with summands of the form $\operatorname{Hom}_A(P(\ell), \Delta(j))$, where $\ell \in Q_0^1\backslash\{v\}$, according to our claim. Then, it follows that
		$$\dim \operatorname{Hom}_A(P(\ell),\Delta(j))=[\Delta(j):L(\ell)]=0.$$
		\item 	Let $v$ be a sink and assume that $v$ is maximal with respect to $\trianglelefteq^e$. Suppose there is an extension of degree $k$ from $\Delta(i)$ to $\Delta(j)$. Then, there is a non-zero homomorphism in the space $\operatorname{Hom}_A(P^k,\Delta(j))$. Note that the indecomposable direct summands of $P^k$ are of the form $P(\ell)$, with $\ell\in Q_0$. Since the $\operatorname{Hom}$-functor is additive, consider the space $\operatorname{Hom}_A(P(\ell),\Delta(j))$. For this space to have a non-zero element, we must have $[\Delta(j):L(\ell)]>0$. The only possibility for such an $\ell$ is $\ell=v$. But then $[\Delta(j):L(v)]>0$, implying $v\triangleleft^e j$, contradicting maximality of $v$.
		
		Assume instead that $v$ is minimal with respect to $\trianglelefteq^e$. For $k=0$, the claim of the proposition is $$\operatorname{Hom}_A(\Delta(i),\Delta(j))=0,$$ which is true because $i\in Q_0^1\backslash \{v\}$ and $j\in Q_0^2 \backslash\{v\}$.
		
	 	Next, we claim that $[P^k:L(v)]=0$ for all $k\geq 1$ and proceed by induction. Consider the following picture:
	 	$$\xymatrixcolsep{0.7cm}\xymatrixrowsep{0.5cm}\xymatrix{
			\dots \ar[r] & P^{1} \ar[rr] \ar@{->>}[rd]& & P(i) \ar[r]^-{\pi_i}& \Delta(i)	\\
			& & \ker \pi_i \ar@{^{(}->}[ur]
		}$$
		Since $A$ is quasi-hereditary, the module $\ker\pi_i$ is contained in $\mathcal{F}(\Delta (\triangleright^e i))$, by which we mean that each standard module $\Delta(a)$ occuring in its filtration satisfies $i \triangleleft^e a$. If 
		$$\dim \operatorname{Ext}_A^1(\Delta(i),\Delta(j))>0,$$ then $\dim \operatorname{Hom}_A(P^1,\Delta(j))>0$, which can only happen when $P^1$ has a direct summand isomorphic to $L(v)$. This implies that $L(v)\subset \operatorname{top} \ker \pi_i$, which means that some $\Delta(b)$ with $[\Delta(b):L(v)]>0$ occurs in the standard filtration of $\ker \pi_i$. Since $L(v)$ is contained in $\operatorname{top} \ker \pi_i$, the module $\Delta(v)$ occurs in its standard filtration, since any other standard module $\Delta(c)$ with a composition factor $L(v)$ must satisfy $L(v)\subset \operatorname{rad}\Delta(c)$. Since $\ker \pi_i \in \mathcal{F}(\Delta(\triangleright^e i))$, this implies $v\triangleleft^e i$, contradicting minimality of $v$.
		
		For the inductive step, assume that the module $P^k$ is contained in $\mathcal{F}(\Delta(\triangleright^e i))$. Then, so is $\ker d_k$, and we argue as above.
		 \qedhere
	\end{enumerate}
\end{proof}
\begin{lemma}\label{lemma:projective resolution of module in A^l does not contain P(v)}
	Assume that $v$ is a source. Let $i\in Q_0^\ell \backslash\{v\}$ and let $P^\bullet \to \Delta^\ell(i)$ be a minimal projective resolution of $\Delta^\ell(i)$ as an $A$-module. Then no term $P^k$ of the projective resolution $P^\bullet$ contains the projective module $P^\ell(v)$ as a direct summand.
\end{lemma}
\begin{proof}
	We proceed by induction. The basis is clear, as $P^0=P^\ell(i)$. Consider the following picture.
	$$\xymatrixcolsep{0.7cm}\xymatrixrowsep{0.5cm}\xymatrix{
		\dots \ar[r]^-{d_{k+2}} & P^{k+1} \ar[rr]^{d_{k+1}} \ar@{->>}[rd]& & P^k \ar[r]^-{d_k}& \dots	\\
		& & \ker d_k \ar@{^{(}->}[ur]
	}$$
	Assume that the statement holds for the module $P^k$ and assume towards a contradiction that the module $P^{k+1}$ contains a direct summand $P^\ell(v)$. By construction, $P^{k+1}$ is a projective cover of $\ker d_k$, thus we conclude that the module $\operatorname{top} \ker d_k$ contains a direct summand $L^\ell(v)$. This fact, in turn, implies that there is a direct summand $P^\ell(t)$ of $P^k$, such that $\left[ P^\ell(t):L^\ell(v)\right]>0$. Since a basis of $P^\ell(t)$ is given by paths in $A^\ell$ starting in $t$, this implies that there is a path in $A^\ell$ from $t$ to $v$. This is our desired contradiction, as $v$ was assumed to be a source.
\end{proof}
\begin{proposition}\label{proposition:ext between standards in deconcatenation}
	Let $Q^1\sqcup Q^2$ be a deconcatenation at a sink or source $v$.
	\begin{enumerate}[(i)]
		\item If $i,j\in Q_0^\ell \backslash\{v\}$, there holds $\operatorname{Ext}_{A^\ell}^k(\Delta^\ell(i),\Delta^\ell(j))\cong  \operatorname{Ext}_A^k(\Delta(i),\Delta(j)),$
		for all $k\geq 0$.
		\item If $v$ is a sink, then $\operatorname{Ext}^k_{A^\ell}(\Delta^\ell(v),\Delta^\ell(j))\cong \operatorname{Ext}_A^k(\Delta(v),\Delta(j))$, for all $k\geq 0$.
		\item If $v$ is a sink, then 
		$\operatorname{Ext}^k_{A^\ell}(\Delta^\ell(i),\Delta^\ell(v))\cong \operatorname{Ext}_A^k(\Delta(i),\Delta(v))$, for all $k\geq 0$.
	\end{enumerate} 
\end{proposition}
\begin{proof}Let $P^\bullet \to \Delta^\ell(i)$ and $Q^\bullet \to \Delta^\ell(j)$ be minimal projective resolutions. By Lemma \ref{lemma:elem. properties of deconcatenation}, there is an isomorphism of $A$-modules $P^\ell(x)\cong P(x)$, for all $x\in Q_0^\ell\backslash\{v\}$. Note that this isomorphism is the one induced by the functor $F^\ell$, that is, $F^\ell(P^\ell(x))\cong  P(x)$. Similarly, we have $F^\ell(\Delta^\ell(i))\cong \Delta(i)$ and $F^\ell(\Delta^\ell(j))\cong \Delta(j)$.
	\begin{enumerate}[(i)]
		\item The functor $F^\ell:A^\ell\operatorname{-mod}\to A\operatorname{-mod}$ is fully faithful and exact. By Lemma \ref{lemma:projective resolution of module in A^l does not contain P(v)}, the terms of the projective resolution $P^\bullet \to \Delta^\ell(i)$ do not contain any direct summands isomorphic to $P^\ell(v)$. We conclude that $F^\ell(P^\bullet)\to F^\ell(\Delta^\ell(i))\cong \Delta(i)$ is a minimal projective resolution. Similarly, we have  that $F^\ell(Q^\bullet) \to F^\ell(\Delta^\ell(j))\cong \Delta(j)$ is a minimal projective resolution. The statement follows.
		\item If $v$ is a sink, we have $P^\ell(v)\cong L(v)\cong P(v)$ as $A$-modules, by Lemma \ref{lemma:elem. properties of deconcatenation}, and $\Delta^\ell(v)\cong L(v)\cong \Delta(v)$ as $A$-modules, by Lemma \ref{lemma:qh-structure on parts from qh-structure of whole}. Thus, the statement follows by applying the functor $F^\ell$ to the minimal projective resolutions of $\Delta^\ell(v)$ and $\Delta^\ell(j)$.
		\item Similar to (ii). \qedhere
	\end{enumerate}
\end{proof}
We recall the notation of \cite{FKR}, where $\overline{1}=2$ and $\overline{2}=1$.
\begin{lemma}\label{lemma:homs into projective at source}
	Let $Q_1\sqcup Q_2$ be a deconcatenation of $Q$ at the source $v$ and let $i\in Q_0^\ell \backslash \{v\}$. Then, there are isomorphisms of vector spaces
	\begin{align*} \operatorname{Hom}_A(P(i),P^\ell(v))\cong \operatorname{Hom}_A(P(i),P(v))\quad \textrm{and}\quad \operatorname{Hom}_A(P^\ell(v),P(i))\cong \operatorname{Hom}_A(P(v),P(i)), \textrm{ for }\ell=1,2.
	\end{align*}
\end{lemma}
\begin{proof}
The first statement is precisely Lemma~\ref{lemma:morphisms lift between from part of deconcatenation} applied to the modules $P(i)$ and $P^\ell(v)$. Note that in the notation of Lemma~\ref{lemma:morphisms lift between from part of deconcatenation}, we have $P^\ell(v)=\overline{P(v)}$. For the second statement, let $p:P(v)\to P^\ell(v)$ denote the natural epimorphism. Then, there is a short exact sequence
$$\xymatrix{0 \ar[r] & \ker p \ar[r] & P(v) \ar[r]^-p & P^\ell(v) \ar[r] & 0}.$$
Note that $P^\ell(v)$ is an $A^\ell$-module while $\ker p$ is an $A^{\overline{\ell}}$-module. Applying $\operatorname{Hom}_A(\blank, P(i))$ to this sequence, we obtain:
$$\xymatrix{
	0\ar[r] & \operatorname{Hom}_A(P^\ell(v), P(i)) \ar[r] & \operatorname{Hom}_A(P(v),P(i)) \ar[r] & \operatorname{Hom}_A(\ker p, P(i))
	}$$
The space $\operatorname{Hom}_A(\ker p, P(i))$ is zero because of our previous observation, so we have
$$\operatorname{Hom}_A(P^\ell(v), P(i))\cong \operatorname{Hom}_A(P(v), P(i))$$
since the Hom-functor is left exact. \qedhere
\end{proof}
\begin{lemma}\label{lemma:projective resolution of standard at source}
	Let $v$ be a source, let $P^\bullet \to \Delta^1(v)$ be a minimal projective resolution with terms $P^k$, $k\geq 0$ and let $Q^\bullet \to \Delta^{2}(v)$ be a minimal projective resolution with terms $Q^k$,  $k\geq 0$. Then, there is a minimal projective resolution $R^\bullet \to \Delta(v)$, with terms $R^k=P^k\oplus Q^k$, for $k\geq 1$, and $R^0=P(v)$.
\end{lemma}
\begin{proof}
	We proceed by induction. For the case $k=1$, consider the projection $p:P(v)\to \Delta(v)$. The module $\ker p$ has a basis consisting of paths $q$ in $A$, such that $s(q)=v$ and $t(q)=j$ with $j\ntriangleleft v$. Let $x_1,\dots, x_n$ be those paths that satisfy $t(x_i)\in Q_0^1$ and let $y_1, \dots, y_m$ be those paths that satisfy $t(y_j)\in Q_0^2$. Put $X=\operatorname{span}\{x_1,\dots, x_n\}$ and $Y=\operatorname{span}\{y_1,\dots, y_m\}$. Clearly, $\ker p=X\oplus Y$ as vector spaces. Since the action of $A$ on $X$ and $Y$ is given by left multiplication, we see that, in fact, $\ker p=X\oplus Y$ as $A$-modules. 
	
	It is clear that, as an $A^1$-module, $X$ is isomorphic to the kernel of the projection $p^1: P^1(v)\twoheadrightarrow \Delta^1(v)$, so $P^1$ is a projective cover of $X$. Similarly, as an $A^2$-module, $Y$ is isomorphic to the kernel of the projection $p^2: P^2(v)\twoheadrightarrow \Delta^2(v)$, so $Q^1$ is a projective cover of $Y$. When $k=2$, we consider the picture:
	$$\xymatrix{
	R^2 \ar[r] \ar[d]&P^1\oplus Q^1 \ar[rd] \ar[rr]^-{(d^1, e^1)} & & P(v) \\
	\ker d^1\oplus \ker e^1 \ar[ru]& & X\oplus Y \ar[ru]	}$$
It is clear that we may view the map from $P^1\oplus Q^1$ to $P(v)$ as $(d^1, e^1)$. It the follows that its kernel is the direct sum $\ker d^1\oplus \ker e^1$, which has a projective cover $R^2=P^2\oplus Q^2$.
	This finishes the basis of the induction. Consider the following picture:
	\begin{align*}
		\xymatrix@C=2cm{
			R^{k+1} \ar[d]\ar[r]^-{f^{k+1}}	&P^k\oplus Q^k \ar[r]^-{\left(\begin{smallmatrix}
					d^k &0 \\ 0 & e^k
				\end{smallmatrix}\right)} & P^{k-1}\oplus Q^{k-1} \\
			\ker d^k\oplus \ker e^k \ar[ru]
		}
	\end{align*}
	That the right matrix above is diagonal follows from Lemma \ref{lemma:no homs between projectives in different parts of deconcatenation} and the induction hypothesis. Then, the kernel of the map $\left( \begin{smallmatrix}
		d^k & 0 \\ 0 & e^k
	\end{smallmatrix}\right)$ is isomorphic to $\ker d^k \oplus \ker e^k$, and a projective cover is therefore $P^{k+1}\oplus Q^{k+1}$. This shows that $R^{k+1}=P^{k+1}\oplus Q^{k+1}$ and $f^{k+1}=\left(\begin{smallmatrix}
	d^{k+1} & 0 \\ 0 & e^{k+1}
\end{smallmatrix}\right)$.\qedhere
\end{proof}
\begin{proposition}\label{proposition:ext to and from source in deconcatenation}
	Let $Q_1\sqcup Q_2$ be a deconcatenation of $Q$ at the source $v$ and let $i\in Q_0^\ell \backslash \{v\}$. Then, there are isomorphisms of vector spaces
	\begin{align*}
		\operatorname{Ext}^k_{A^\ell}(\Delta^\ell(i),\Delta^\ell(v))&\cong \operatorname{Ext}^k_{A}(\Delta(i),\Delta(v))\quad \textrm{and}\quad  		\operatorname{Ext}^k_{A^\ell}(\Delta^\ell(v),\Delta^\ell(i))\cong \operatorname{Ext}^k_{A}(\Delta(v),\Delta(i)),
	\end{align*}
	for all $k\geq 0$.
\end{proposition}
\begin{proof}
	Let $P^\bullet\to \Delta^\ell(i)$ and $Q^\bullet \to \Delta^\ell(v)$ be minimal projective resolutions. Let $\varepsilon^\ell: P^\bullet \to Q^\bullet[k]$ be a chain map.
	\begin{align*}
		\xymatrix{
			\dots \ar[r]& P^{k+2}\ar[r]^{d^{k+2}} \ar[d]^-{\varepsilon_2^\ell} & P^{k+1} \ar[d]^-{\varepsilon_1^\ell}\ar[r]^{d^{k+1}}& P^k	\ar[d]^-{\varepsilon_0^\ell}	\ar[r] & \dots\\
			\dots \ar[r]& Q^{2}\ar[r]_-{\delta^2} & Q^{1}\ar[r]_-{\delta^1} & P^\ell(v)
		}
	\end{align*}
	According to Lemma \ref{lemma:projective resolution of standard at source}, there is a minimal projective resolution $T^\bullet\to \Delta(v)$ with terms
	$$T^k=\begin{cases}
		Q^k \oplus R^k, & \textrm{if }k\geq 1\\
		P(v), & \textrm{if } k=0,
	\end{cases}$$
where the modules $R^k$ constitute a minimal projective resolution of $\Delta^{\overline{\ell}}(v)$. Define a map $\varepsilon: P^\bullet \to T^\bullet[k]$, as having components given by
	$$\varepsilon_n=\begin{cases}
		\left(\begin{smallmatrix}
			\varepsilon_n^\ell \\ 0
		\end{smallmatrix}\right), & \textrm{if } n\geq 1 \\
		\Phi(\varepsilon_0^\ell), & \textrm{if }n=0.
	\end{cases}$$
Here, $\Phi: \operatorname{Hom}_A(P^k, P^\ell(v))\to \operatorname{Hom}_A(P^k, P(v))$ is the isomorphism obtained in the proof of Lemma \ref{lemma:homs into projective at source}, so that $\varepsilon_0$ is uniquely defined by the equation $p\varepsilon_0=\varepsilon_0^\ell$, where $p:P(v)\to P(v)^\ell$ is the natural projection.
	\begin{align*}
		\xymatrix{
			\dots \ar[r]& P^{k+2}\ar[r]^{d^{k+2}} \ar[d]^-{\varepsilon_2} & P^{k+1} \ar[d]^-{\varepsilon_1}\ar[r]^{d^{k+1}}& P^k	\ar[d]^-{\varepsilon_0}	\ar[r] & \dots\\
			\dots \ar[r]& Q^{2}\oplus R^2\ar[r]_-{\left(\begin{smallmatrix}
					\delta^2 & 0 \\ 0 & \partial^2
				\end{smallmatrix}\right)} & Q^{1} \oplus R^1\ar[r] & P(v)
		}
	\end{align*}
We claim that $\varepsilon$ is a chain map. Considering the following diagram, where $n\geq 1$.
$$\xymatrixcolsep{2cm}\xymatrix{
	P^{k+n+1}\ar[d]_-{\varepsilon_{n+1}} \ar[r]^-{d^{k+n+1}} & P^{k+n} \ar[d]^-{\varepsilon_n} \\
	Q^{n+1} \oplus R^{n+1} \ar[r]_-{\left(\begin{smallmatrix}
			\delta^{n+1} & 0 \\ 0 & \partial^{n+1}
		\end{smallmatrix}\right)}& Q^n\oplus R^n	
}$$
We have
\begin{align*}
\left(\begin{smallmatrix}
	\delta^{n+1} & 0 \\ 0 & \partial^{n+1}
\end{smallmatrix}\right) \varepsilon_{n+1}&=\left(\begin{smallmatrix}
\delta^{n+1} & 0 \\ 0 & \partial^{n+1}
\end{smallmatrix}\right)\left(\begin{smallmatrix}
\varepsilon_{n+1}^\ell \\ 0
\end{smallmatrix}\right)=\left(\begin{smallmatrix}
\delta^{n+1}\varepsilon_{n+1}^\ell \\0
\end{smallmatrix}\right)=\left(\begin{smallmatrix}
\varepsilon_n^\ell d^{k+n+1} \\ 0
\end{smallmatrix}\right)=\left(\begin{smallmatrix}
	\varepsilon_n^\ell\\0
\end{smallmatrix}\right) d^{k+n+1}=\varepsilon_n d^{k+n+1}
\end{align*}
since $\delta^{n+1}\varepsilon_{n+1}^\ell=\varepsilon_n^\ell d^{k+n+1}$ by virtue of $\varepsilon^\ell$ being a chain map. It remains to check the following square.
	$$\xymatrix{
		P^{k+1} \ar[d]_-{\varepsilon_1}\ar[r]^{d^{k+1}} & P^k	\ar[d]^-{\varepsilon_0}	\\
		Q^{1} \oplus R^1\ar[r] & P(v)
	}$$
	In the proof of Lemma \ref{lemma:projective resolution of standard at source}, we saw that the kernel of the projection $P(v)\twoheadrightarrow \Delta(v)$ is isomorphic to $X\oplus Y$, where $X$ is isomorphic to the kernel of the projection $P^1(v)\twoheadrightarrow \Delta^1(v)$ and $Y$ is isomorphic to the kernel of the projection $P^2(v)\twoheadrightarrow \Delta^2(v)$. Consider the following pictures:
	$$\xymatrix{
		Q^1 \ar[rr]^-{\delta^1} \ar[rd]& & P^1(v) \ar[r] & \Delta^1(v)\\
		& X \ar[ru]
	},\quad \xymatrix{
		R^1 \ar[rr]^-{\partial^1}\ar[rd]& & P^2(v) \ar[r] & \Delta^2(v)\\
		& Y \ar[ru]
	}$$
	By construction, $\delta^1$ and $\partial^1$ are radical maps. Since $\operatorname{rad}P^1(v)$ and $\operatorname{rad}P^2(v)$ are isomorphic to submodules of $P(v)$, we may view $\delta^1$ and $\partial^1$ as maps $\delta^1: Q^1\to P(v)$ and $\partial^1: R^1\to P(v)$, respectively. Then the bottom map in the above square can be written as $(\delta^1,\partial^1):P^1 \oplus R^1\to P(v)$. Consider the following picture:
	$$\xymatrix{
		P^{k+1} \ar[d]_-{\varepsilon_1} \ar[r]^-{d^{k+1}} & P^k \ar[r]^-{\varepsilon_0^\ell} \ar[d]^-{\varepsilon_0} & P^\ell(v)\\
		Q^1\oplus R^1 \ar[r]_-{(\delta^1,\partial^1)} & P(v) \ar@{->>}[ru]_-{p}
	}$$
	By construction, the triangle on the right commutes. Since $\varepsilon^\ell$ is a chain map, we have
	$$(\delta^1,\partial^1)\varepsilon_1=(\delta^1, \partial^1)\left(\begin{smallmatrix}
		\varepsilon_1^\ell \\ 0
	\end{smallmatrix}\right)=\delta^1 \varepsilon_1^\ell=\varepsilon_0^\ell d^{k+1}.$$
	Moreover, the image of $\delta^1\varepsilon_1^\ell=\varepsilon_0^\ell d^{k+1}$ is contained in the submodule $\operatorname{rad}(A^\ell)\cdot P(v)$, on which $p$ acts as the identity. This implies that the perimeter of the diagram commutes, and since $p$ is an isomorphism on $\operatorname{rad}(A^\ell)\cdot P(v)$, the square commutes. Let $e^\ell:P^\bullet \to Q^\bullet[k]$ be another chain map. We claim that if $\varepsilon^\ell$ and $e^\ell$ are homotopic, so are $\varepsilon$ and $e$, where $e: P^\bullet \to T^\bullet[k]$ is the chain map with components
	$$e_n=\begin{cases}
		\left(\begin{smallmatrix}
			e_n^\ell \\ 0
		\end{smallmatrix}\right), & \textrm{if } n\geq 1 \\
		\Phi(e_0^\ell), & \textrm{if }n=0.
	\end{cases}$$
	$$\xymatrixcolsep{2cm}\xymatrix{
		P^{k+n+1}\ar[r]^-{d^{k+n+1}}  \ar[d] & P^{k+n} \ar[ld]_-{h_{n+1}}\ar[d]^-{\varepsilon_n^\ell}_-{e_n^\ell} \ar[r]^-{d^{k+n}} &P^{k+n-1}  \ar[d]\ar[ld]_-{h_n}\\
		Q^{n+1} \ar[r]_-{\delta^{n+1}} & Q^n \ar[r]_-{\delta^n}& Q^{n-1}
	}$$
	Assume that $h_n d^{k+n}+\delta^{n+1}\ h_{n+1} =\varepsilon_n^\ell-e_n^\ell$, for $n\geq 1$. Consider the picture:
	$$\xymatrixcolsep{2cm}\xymatrixrowsep{1.5cm}\xymatrix{
		P^{k+n+1}\ar[r]^-{d^{k+n+1}}  \ar[d] & P^{k+n}\ar[d]^-{\left(\begin{smallmatrix}
				\varepsilon_n^\ell\\ 0
			\end{smallmatrix}\right)}_-{\left(\begin{smallmatrix}
				e_n^\ell\\ 0
			\end{smallmatrix}\right)} \ar[r]^-{d^{k+n}} \ar[ld]_-{\left(\begin{smallmatrix}
				h_{n+1}\\0
			\end{smallmatrix}\right)}&P^{k+n-1} \ar[ld]_-{\left(\begin{smallmatrix}
				h_{n}\\0
			\end{smallmatrix}\right)}  \ar[d]\\
		Q^{n+1}\oplus R^{n+1} \ar[r]_-{\left(\begin{smallmatrix}
				\delta^{n+1} & 0 \\ 0 & \partial^{n+1}
			\end{smallmatrix}\right)} & Q^n\oplus R^n \ar[r]_-{\left(\begin{smallmatrix}
				\delta^{n} & 0 \\ 0 & \partial^{n}
			\end{smallmatrix}\right)}& Q^{n-1}\oplus R^{n-1}
	}$$
	Then, we have
	$$\left(\begin{smallmatrix}
		\delta^{n+1} & 0 \\ 0 & \partial^{n+1}
	\end{smallmatrix}\right)  \left(\begin{smallmatrix}
		h_{n+1}\\0
	\end{smallmatrix}\right) + \left(\begin{smallmatrix}
		h_n\\0
	\end{smallmatrix}\right)d^{k+n}=\left(\begin{smallmatrix}
		\delta^{n+1} h_{n+1}\\0
	\end{smallmatrix}\right) + \left(\begin{smallmatrix}
		h_n d^{k+n}\\0
	\end{smallmatrix}\right) =\left(\begin{smallmatrix}
		\varepsilon_n^\ell-e_n^\ell\\0
	\end{smallmatrix}\right).$$
It remains to check the case $n=0$. Consider the following picture.
	$$\xymatrixcolsep{2cm}\xymatrix{
		P^{k+1}\ar[r]^-{d^{k+1}}  \ar[d] & P^{k} \ar[ld]_-{h_{1}}\ar[d]^-{\varepsilon_0^\ell}_-{e_0^\ell} \ar[r]^-{d^{k}} &P^{k-1}  \ar[d]\ar[ld]_-{h_0}\\
		Q^{1} \ar[r]_-{\delta^{1}} & P^\ell(v) \ar[r]& 0
	}$$
	Assume that $h_0\circ d^k + \delta^1 h_1=\varepsilon_0^\ell-e_0^\ell$ as maps $P^k\to P^\ell(v)$. To finish our homotopy, we take the map $\Phi(h_0)$, where $$\Phi:\operatorname{Hom}_A(P^k, P^\ell(v))\to \operatorname{Hom}_A(P^k, P(v))$$
	is the isomorphism constructed in the proof of Lemma \ref{lemma:homs into projective at source}.
	$$\xymatrixcolsep{2cm}\xymatrixrowsep{1.5cm}\xymatrix{
		P^{k+1} \ar[d]\ar[r]^-{d^{k+1}} & P^k \ar[ld]_-{\left(\begin{smallmatrix}
				h_1\\0
			\end{smallmatrix}\right)} \ar[r]^-{d^k} \ar[d]^-{\varepsilon_0}_-{e_0} & P^{k-1}\ar[ld]_-{\Phi(h_0)} \ar[d]^-{h_0}\\
		Q^1\oplus R^1 \ar[r]_-{(\delta^1,\partial^1)} & P(v) \ar@{->>}[r]_-{p} & P^\ell(v)
	}$$
Note that the maps $\varepsilon_0, e_0$ and $\Phi(h_0)$ are defined by the equations
$$p\varepsilon_0=\varepsilon_0^\ell,\quad pe_0=e_0^\ell,\quad\textrm{and}\quad p\Phi(h_0)=h_0.$$
Using this, we obtain $\delta^1h_1 + p\Phi(h_0)d^k=p(\varepsilon_0-e_0)$. Since the image of $\delta^1 h_1$ is contained in the submodule $\operatorname{rad}P^\ell(v)$, which in a natural way is a submodule of $P(v)$ on which $p:P(v)\to P^\ell(v)$ acts as the identity, we have $\delta^1 h_1=p\delta^1 h_1$ as maps $P^k\to P^\ell(v)$. This means that we have the equality
$$p(\delta^1h_1 + \Phi(h_0)d^k)=p(\varepsilon_0- e_0)$$
as maps $P^k\to P^\ell(v)$. Now, since $P^k$ is an $A^\ell$-module, we may argue as in the proof of Lemma \ref{lemma:homs into projective at source}, to conclude that the images of $\delta^1 h_1, \Phi(h_0), \varepsilon_0$ and $e_0$ are all contained in the submodule $\operatorname{rad}(A^\ell)\cdot P(v)\subset P(v)$, on which $p$ acts as the identity. This implies that
$$\varepsilon_0 - e_0= \delta^1h_1 + \Phi(h_0)d^k= (\delta^1,\partial^1)\left(\begin{smallmatrix}
	h_1 \\0
\end{smallmatrix}\right)+\Phi(h_0)d^k,$$
confirming the claim that $\varepsilon$ and $e$ are homotopic. Thus, we have shown that the map
$$\Upsilon:\operatorname{Ext}_{A^\ell}^k(\Delta^\ell(i),\Delta^\ell(v)) \to \operatorname{Ext}_A^k(\Delta(i),\Delta(v)),$$
defined by $\varepsilon^\ell \mapsto \varepsilon$, is well-defined.
	
	Next, note that, since $i\in Q_0^\ell\backslash\{v\}$, the terms of $P^\bullet$ do not contain indecomposable direct summands isomorphic to $P^\ell(v)$, according to Lemma \ref{lemma:projective resolution of module in A^l does not contain P(v)}. Then, since the indecomposable direct summands of $R^s$ are of the form $P(j)$, for $j\in Q_0^{\overline{\ell}}$, we have $\operatorname{Hom}_A(P^{k+s}, R^s)=0$, for any $s\geq 1$, according to Lemma \ref{lemma:no homs between projectives in different parts of deconcatenation}. Thus, we have an isomorphism of vector spaces
	$$\operatorname{Hom}_{A^\ell} (P^{k+s}, Q^s) \cong \operatorname{Hom}_A(P^{k+s}, T^s),$$
	for all $s\geq 1$, given by $f\mapsto \left(\begin{smallmatrix}
		f \\ 0
	\end{smallmatrix}\right)$ (note that $T^s=Q^s\oplus R^s$ for $s\geq 1$). When $s=0$, we have the isomorphism $\operatorname{Hom}_{A^\ell}(P^k, Q^0)\cong \operatorname{Hom}_A(P^k, T^0)$, according to Lemma \ref{lemma:homs into projective at source}, since $Q^0=P^\ell(v)$ and $T^0=P(v)$. This means that we have an isomorphism of vector spaces
$$\xi^s: \operatorname{Hom}_{A^\ell}(P^{k+s}, Q^s)\cong \operatorname{Hom}_{A}(P^{k+s}, T^s),$$
	for all $s\geq 0$. We note that, by construction, $\left(\Upsilon(\varepsilon^\ell)\right)_n=\xi^n(\varepsilon^\ell_n)$ for all $n\geq 0$, which shows that $\Upsilon$ is a linear isomorphism. This proves the first statement. The second statement is proved similarly.
\end{proof}
Let $\operatorname{Ext}_{A^\ell}^\ast(\Delta^\ell,\Delta^\ell)$ denote the $\operatorname{Ext}$-algebra of standard modules over $A^\ell$, for $\ell=1,2$. Similarly, let $\operatorname{Ext}_A^\ast(\Delta,\Delta)$ denote the $\operatorname{Ext}$-algebra of standard modules over $A$. Fix the basis $1_{\Delta(v)}$ of the space $\operatorname{End}_A(\Delta(v))$. 
\begin{theorem}\label{theorem:pasting of ext-algebras of parts of concatenation isomorphic to whole ext-algebra}
	Let $Q^1\sqcup Q^2$ be a deconcatenation of $Q$ at a sink or source $v$, such that if $v$ is a sink, then $v$ is minimal or maximal with respect to $\trianglelefteq^e$. Then, there is an isomorphism of graded algebras $\operatorname{Ext}_A^\ast(\Delta,\Delta)\cong \operatorname{Ext}_{A^1}^\ast(\Delta^1,\Delta^1)\diamond \operatorname{Ext}_{A^2}^\ast(\Delta^2,\Delta^2)$.
\end{theorem}
\begin{proof}
	We start by noting that, according to Propositions \ref{proposition:ext between standards in deconcatenation} and \ref{proposition:ext to and from source in deconcatenation}, there holds $$\operatorname{Ext}^k_{A^\ell}(\Delta^\ell(i),\Delta^\ell(j))\cong \operatorname{Ext}_A^k(\Delta(i),\Delta(j)),$$ for all $i,j\in Q_0^\ell$ and all $k\geq 0$. Put $C=\operatorname{Ext}_{A^1}^\ast(\Delta^1,\Delta^1)\diamond \operatorname{Ext}_{A^2}^\ast(\Delta^2,\Delta^2).$
By definition of the multiplication on $C$, for any $x\in \operatorname{Ext}_{A^1}^\ast(\Delta^1,\Delta^1)$ and $y\in \operatorname{Ext}_{A^2}^\ast(\Delta^2,\Delta^2)$ such that $\deg x>0$ and $\deg y>0$, there holds $x\cdot_C y=y\cdot_C x=0$. Similarly, according to Proposition \ref{proposition:no ext between different parts of concatenation}, we have $$\operatorname{Ext}_{A^1}^{>0}(\Delta^1,\Delta^1)\cdot \operatorname{Ext}_{A^2}^{>0}(\Delta^2,\Delta^2)=\operatorname{Ext}_{A^2}^{>0}(\Delta^2,\Delta^2)\cdot\operatorname{Ext}_{A^1}^{>0}(\Delta^1,\Delta^1)=0.$$
	Together, these facts imply that there is a homomorphism of graded algebras $\Phi:C\to \operatorname{Ext}_A^\ast(\Delta,\Delta)$. Note that, because of Proposition \ref{proposition:no ext between different parts of concatenation}, the subspaces $\Phi(\operatorname{Ext}_{A^\ell}^\ast(\Delta^\ell,\Delta^\ell))$ generate $\operatorname{Ext}_A^\ast(\Delta,\Delta)$, showing that $\Phi$ is surjective. Lastly, note that $\dim_K C=\dim_K \operatorname{Ext}_A^\ast(\Delta,\Delta)$, so that $\Phi$ is a surjective linear map between vector spaces of the same dimension, hence a linear isomorphism.
\end{proof}
\begin{example}
	We consider the deconcatenation $(\xymatrixcolsep{0.5cm}\xymatrix{1 & 2 \ar[l] & 3\ar[l]}) \sqcup (\xymatrixcolsep{0.5cm}\xymatrix{3 \ar[r] &4 \ar[r] & 5})$
	of the quiver 
	$$Q=\xymatrixcolsep{0.5cm}\xymatrix{1 & 2 \ar[l]& 3\ar[r] \ar[l] & 4\ar[r] & 5}.$$
	We draw the Loewy diagrams of the projective and standard modules over $A^1$.
	\begin{align*}
		P^1(1)&: 1,\quad P^1(2): \vcenter{\xymatrixrowsep{0.5cm}\xymatrix{2\ar[d]\\1}},\quad P^1(3):\vcenter{\xymatrixrowsep{0.5cm}\xymatrix{3\ar[d]\\2\ar[d]\\1}},\quad \Delta^1(1)\cong P^1(1),\quad \Delta^1(2)\cong P^1(2),\quad \Delta^1(3)\cong P^1(3).
	\end{align*}
	As the standard modules over $A^1$ are projective, there are no non-split extensions between them. We see that there are homomorphisms $f:\Delta^1(1)\to \Delta^1(2)$, $g:\Delta^1(2)\to \Delta^1(3)$ and $h:\Delta^1(1)\to \Delta^1(3)$, such that $h=g\circ f$. The quiver of the Ext-algebra $\operatorname{Ext}_{A_1}^\ast(\Delta^1,\Delta^1)$ is then
	$$\xymatrix{1 \ar@{-->}[r]^-f & 2 \ar@{-->}[r]^-g & 3},$$
	with $\deg f=\deg g=0$, subject to no additional relations.
	
	Next, we draw the Loewy diagrams of the projective and standard modules over $A^2$. 
	\begin{align*}
		P^2(3)&: \vcenter{\xymatrixrowsep{0.5cm}\xymatrix{3\ar[d]\\4\ar[d]\\5}},\quad P^2(4): \vcenter{\xymatrixrowsep{0.5cm}\xymatrix{4\ar[d]\\5}},\quad P^2(5):5,\quad \Delta^2(3):3,\quad \Delta^2(4):4,\quad \Delta^2(5):5.
	\end{align*} As the standard modules over $A^2$ are simple, we see that we have non-split extensions $\alpha\in \operatorname{Ext}_{A_2}^1(\Delta^2(3),\Delta^2(4))$ and $\beta\in \operatorname{Ext}_{A_2}^1(\Delta^2(4),\Delta^2(5))$, such that $\beta\alpha=0$. The quiver of the Ext-algebra $\operatorname{Ext}_{A_2}^\ast(\Delta^2,\Delta^2)$ is then
	
	$$\xymatrix{3 \ar[r]^-\alpha & 4 \ar[r]^-\beta& 5},$$
	subject to the relation $\beta\alpha=0$. Finally, we draw the Loewy diagrams of the projective and standard modules over $A$.
	\begin{align*}
		P(1)&:\vcenter{\xymatrixrowsep{0.5cm}\xymatrix{1}},\quad P(2):\vcenter{\xymatrixrowsep{0.5cm}\xymatrix{2\ar[d]\\1}},\quad P(3):\vcenter{\xymatrixcolsep{0.5cm}\xymatrixrowsep{0.5cm}\xymatrix{&3 \ar[ld]\ar[rd]\\2\ar[d] & & 4\ar[d]\\1 & & 5}},\quad P(4):\vcenter{\xymatrixrowsep{0.5cm}\xymatrix{4 \ar[d]\\5}},\quad P(5):\vcenter{\xymatrixrowsep{0.5cm}\xymatrix{5}}, \\
		\Delta(1)&\cong P(1),\quad \Delta(2)\cong P(2),\quad \Delta(3):\vcenter{\xymatrixrowsep{0.5cm}\xymatrix{3\ar[d]\\2\ar[d]\\1}},\quad \Delta(4)\cong L(4),\quad \Delta(5)\cong P(5).
	\end{align*}
	We see that there are homomorphisms $f^\prime:\Delta(1)\to \Delta(2)$, $g^\prime: \Delta(2) \to \Delta(3)$ and $h^\prime: \Delta(1)\to \Delta(3)$, such that $h^\prime=g^\prime\circ f^\prime$. Additionally, we see that there are non-split extensions $\alpha^\prime\in \operatorname{Ext}_A^1(\Delta(3),\Delta(4))$ and $\beta^\prime\in \operatorname{Ext}_{A}^1(\Delta(4),\Delta(5))$, such that $\beta^\prime \alpha^\prime=0$. The composition $\alpha^\prime g^\prime$ produces an extension in $\operatorname{Ext}_A^1(\Delta(2),\Delta(4))$, and is therefore equal to zero, according to Lemma \ref{proposition:ext between standards in deconcatenation}. Therefore, the Ext-algebra $\operatorname{Ext}_A^\ast(\Delta,\Delta)$ is given by the quiver
	$$\xymatrix{1 \ar@{-->}[r]^-{f^\prime} & 2 \ar@{-->}[r]^-{g^\prime} & 3\ar[r]^-{\alpha^\prime} &4 \ar[r]^-{\beta^\prime} & 5},$$
	subject to the relations $\beta^\prime \alpha^\prime=0$ and $\alpha^\prime g^\prime=0$. We observe that this algebra coincides with the algebra 
	$$C=\operatorname{Ext}_{A^1}^\ast(\Delta^1,\Delta^1)\diamond\operatorname{Ext}_{A^2}^\ast(\Delta^2,\Delta^2).$$
\end{example}
\subsection{Exact Borel subalgebras under deconcatenations}
Consider a deconcatenation $Q^1\sqcup Q^2$ of $Q$ at the sink or source $v$. Let $B^1\subset A^1$ be a subalgebra having a basis $\{b_1,\dots, b_n, e_v\}$ and let $B^2\subset A^2$ be a subalgebra with basis $\{c_1,\dots, c_m, e_v\}$. Then $B=B^1\diamond B^2$ is a subalgebra of $A$.
\begin{proposition}\label{proposition:gluing at source admits borel}
	Let $Q=Q^1\sqcup Q^2$ be a deconcatenation at the source $v$ and assume that we have regular exact Borel subalgebras $B^1\subset A^1$ and $B^2\subset A^2$. Then, $A$ admits a regular exact Borel subalgebra.
\end{proposition}
\begin{proof}
		For any $i\in Q_0^\ell \backslash\{v\}$, we have $\Delta(i)\cong \Delta^\ell(i)$. Let $0\subset M_n\subset \dots \subset M_0=\operatorname{rad}\Delta^\ell(i)$ be a filtration with costandard subquotients. We claim that this is also a filtration of $\operatorname{rad}\Delta(i)$. Consider a subquotient of the filtration
	$$\faktor{M_k}{M_{k+1}}\cong \nabla^\ell(j).$$
	Then, $\nabla^\ell(j)\cong \nabla(j)$ for all $j\in Q_0^\ell$. For $j\neq v$, we appeal to Lemma \ref{lemma:qh-structure on parts from qh-structure of whole}. For $j=v$, this is automatic as $\nabla(v)$ is simple by virtue of $v$ being a source. Left to check is that we have a $\nabla$-filtration of $\operatorname{rad}\Delta(v)$. Since $v$ is a source, we have a direct sum decomposition:
	$$\operatorname{rad}\Delta(v)=\operatorname{rad}\Delta^1(v)\oplus \operatorname{rad}\Delta^2(v).$$
	Since $A^1$ has a regular exact Borel subalgebra, $\operatorname{rad}\Delta^1(v)$ has a filtration whose subquotients are members of the set
	$$\{\nabla(i)\ |\ i\in Q_0^1\backslash\{v\} \}.$$
	To see that $\nabla(v)$ does not occur as a composition factor, note that $[\operatorname{rad}\Delta(v):L(v)]=0$. Indeed, an occurence of a composition factor $L(v)$ would imply that $[P(v):L(v)]\geq 2$, contradicting the fact that $v$ is a source. Similarly, $\operatorname{rad}\Delta^2(v)$ has a filtration whose subquotients are members of the set
	$$\{\nabla(i)\ | \ i\in Q_0^2\backslash\{v\}\}.$$
	We conclude that $\operatorname{rad}\Delta(v)\in \mathcal{F}(\nabla)$, proving the statement.
	\end{proof}
\begin{proposition}\label{proposition:gluing of deconcatenation admits regular exact borel}
	Let $Q=Q^1\sqcup Q^2$ be a deconcatenation at the sink $v$. Assume further that we have regular exact Borel subalgebras $B^1 \subset A^1$, $B^2\subset A^2$ and that the vertex $v$ is minimal with respect to the essential order on $Q_0$. Then, there exists a regular exact Borel subalgebra $B\subset A$.
\end{proposition}
\begin{proof}
	For any $i\in Q_0^\ell \backslash\{v\}$, we have $\Delta(i)\cong \Delta^\ell(i)$. Let $0\subset M_n\subset \dots \subset M_0=\operatorname{rad}\Delta^\ell(i)$ be a filtration with costandard subquotients. We claim that this is also a filtration of $\operatorname{rad}\Delta(i)$. Consider a subquotient of the filtration
	$$\faktor{M_k}{M_{k+1}}\cong \nabla^\ell(j).$$
	Then $\nabla^\ell(j)\cong \nabla(j)$ for all $j\in Q_0^\ell\backslash\{v\}$, according to Lemma \ref{lemma:qh-structure on parts from qh-structure of whole}. For $j=v$, we claim that when $v$ is a sink, $v$ is minimal in the essential order on $Q_0$ if and only if $\nabla(v)$ is simple. Indeed, if $v$ is minimal, then $\nabla(v)$ is clearly simple. Conversely, assume that $\nabla(v)$ is simple and that there exists another vertex $k$ such that $k\triangleleft^e v$. Then, we have
	$$[\Delta(v):L(k)]>0 \quad \textrm{or} \quad (P(k):\Delta(v))>0.$$
	The first condition cannot be satisfied, since $\Delta(v)$ is simple by virtue of $v$ being a sink. According to \cite[Lemma 2.5]{DlabRingel}, we have
	$$(P(k):\Delta(v))=[\nabla(v):L(k)],$$
		so this condition may not be satisfied either, since $\nabla(v)$ was assumed to be simple.
		Therefore, for $i\in Q_0^\ell\backslash\{v\}$, we have $\operatorname{rad}\Delta(i)\in \mathcal{F}(\nabla)$, since $A^1$ and $A^2$ admit regular exact Borel subalgebras, by Theorem \ref{theorem:when does KA_n have regular exact borel}. Since $v$ is a sink, $\Delta(v)\cong L(v)$ which implies that $\operatorname{rad}\Delta(v)=0$, so there is nothing more to show.
\end{proof}
\begin{lemma}\label{lemma:costandard at sink only one with composition factor}
	Let $Q=Q^1\sqcup Q^2$ be a deconcatenation at the sink $v$. Then, for all $i\in Q_0$, we have $\left[\nabla(i):L(v)\right]>0$ if and only if $i=v$.
\end{lemma}
\begin{proof}
	By definition, $\nabla(i)$ is a submodule of $I(i)$, so $\left[\nabla(i):L(v)\right]>0$ if and only if there is a path from $v$ to $i$, not passing through a vertex greater than $i$ in $Q_0$. This is a contradiction, since $v$ is assumed to be a sink.
\end{proof}
From Lemma \ref{lemma:costandard at sink only one with composition factor}, it is also clear that for any module $M\in \mathcal{F}(\nabla)$, the costandard module $\nabla(v)$ appears as a subquotient in the filtration of $M$ if and only if $\left[M:L(v)\right]>0$.

Now we are ready to provide a necessary and sufficient condition for $A$ to admit a regular exact Borel subalgebra $B\subset A$, where $A=KQ$ and $Q=Q^1\sqcup Q^2$ is a deconcatenation at a sink $v$.

\begin{proposition}\label{proposition:sufficient and necessary condition for borel when v is a sink and costandard at v not simple}
	Let $Q=Q^1\sqcup Q^2$ be a deconcatenation at the sink $v$ and assume that $B^1\subset A^1$ and $B^2\subset A^2$ are regular exact Borel subalgebras. Moreover, assume that $v$ is not minimal with respect to the essential order on $Q_0$. Then, the following are equivalent. 
	\begin{enumerate}[(i)]
		\item There exists a regular exact Borel subalgebra $C\subset A$.
		\item $\left[\operatorname{rad}\Delta(i):L(v)\right]=0$, for all $i\in Q_0$.
		\item The vertex $v$ is maximal with respect to the essential order on $Q_0$.
	\end{enumerate}
\end{proposition}
\begin{proof}
	We first show the equivalence (i) $\iff$ (ii). We assume that (ii) holds and claim that $\operatorname{rad}\Delta(i)\in \mathcal{F}(\nabla)$, for all $i\in Q_0$. If $i=v$, we have $\Delta(v)\cong L(v)$ and consequently $\operatorname{rad}\Delta(v)=0$, so there is nothing to prove. Assume that $i\neq v$, where $i\in Q_0^\ell$. Then, $\Delta^\ell(i)\cong \Delta(i)$. According to the remark after Lemma \ref{lemma:costandard at sink only one with composition factor}, any $\nabla$-filtration of $\operatorname{rad}\Delta^\ell(i)$ does not contain $\nabla^\ell(v)$ as a subquotient. Therefore, any $\nabla$-filtration of $\operatorname{rad}\Delta^\ell (i)$ is automatically a $\nabla$-filtration of $\operatorname{rad}\Delta(i)$, since all occurring subquotients in the filtration of $\operatorname{rad}\Delta^\ell(i)$ are of the form $\nabla^\ell(j)$ with $j\neq v$.
	
	Conversely, if $C\subset A$ is a regular exact Borel subalgebra, then $\operatorname{rad}\Delta(i)\in \mathcal{F}(\nabla)$ for $1\leq i\leq n$. Let $i$ be such that $[\operatorname{rad}\Delta(i):L(v)]>0$ and assume without loss of generality that $i\in Q_0^1\backslash\{v\}$. Then, $\nabla(v)$ occurs in the $\nabla$-filtration of $\operatorname{rad}\Delta(i)$, according to the remark after Lemma~\ref{lemma:costandard at sink only one with composition factor}. Note that since $v$ is assumed not to be minimal in the essential order, $\nabla(v)$ is not simple.  Since $Q^2$ is connected, $v$ has at least one neighbour, $j\in Q_0^2\backslash\{v\}$. If $j\triangleleft^e v$, then $L(j)$ is a composition factor in $\nabla(v)$, which implies that $[\operatorname{rad}\Delta(i):L(j)]>0$ since $\nabla(v)$ occurs in the $\nabla$-filtration of $\operatorname{rad}\Delta(i)$. This is a contradiction, since $\operatorname{rad}\Delta(i)$ is an $A^1$-module. If $v\triangleleft^e j$, then $[\operatorname{rad}\Delta(j):L(v)]>0$, so that $\nabla(v)$ is a subquotient of $\operatorname{rad} \Delta(j)$. This contradicts $\nabla(v)$ being a subquotient of $\operatorname{rad}\Delta(i)$, unless $\nabla(v)$ is simple. However, we assumed that $v$ is not minimal in the essential order, guaranteeing it is not simple. If $v$ and $j$ are incomparable, we consider the module $M$, given by the following Loewy diagram:
	$$\xymatrixrowsep{0.4cm}\xymatrix{ j \ar[d] \\ v}$$
	Since $A^2$ is quasi-hereditary, $\trianglelefteq^e$ is adapted to $A^2$, according to Lemma~\ref{lemma: q.h algebra has adapted order}. Then, by definition, there is a vertex $k$ such that $v\triangleleft^e k$, $j\triangleleft^e k$ and $[M:L(k)]\neq 0$. This is a contradiction, since $M$ has no other composition factors besides $L(v)$ and $L(j)$. This proves (i) $\iff$ (ii).
	
	Next, we prove (ii) $\iff$ (iii). Assume (ii) and that $v$ is not maximal. Then there is another vertex $k$, such that $v\triangleleft^ek$. Then, by definition of the essential order, we have
	$$\left[\Delta(k):L(v)\right]>0\quad \textrm{or}\quad (P(v):\Delta(k))>0.$$
	The first condition can not be satisfied, as it contradicts (ii), and neither may the second, since $P(v)$ is simple by virtue of $v$ being a sink. Conversely, assume that $v$ is maximal and consider $P(i)$, for $i\in Q_0$. If $\left[P(i):L(v)\right]=0$, there is nothing to show, as $\operatorname{rad}\Delta(i)$ is a submodule of a quotient of $P(i)$. Otherwise, if $\left[P(i):L(v)\right]>0$, there still holds $\left[\Delta(i):L(v)\right]=0$, since by definition, the composition factors $L(t)$ of $\Delta(i)$ satisfy $t\triangleleft i$.
\end{proof}
\begin{proposition}\label{proposition:gluing of borels is borel}
	Let $Q=Q^1\sqcup Q^2$ be a deconcatenation at a sink or a source $v$. Assume that there are regular exact Borel subalgebras $B^1\subset A^1$ and $B^2\subset A^2$ containing the sets of idempotents $\{e_i \ |\ i\in Q_0^1\}$ and $\{e_j \ |\ j\in Q_0^2\}$, respectively, and that $A$ admits a regular exact Borel subalgebra. Then $B^1\diamond B^2\subset A$ is a regular exact Borel subalgebra of $A$.
\end{proposition}
\begin{proof}
	Consider the regular exact Borel subalgebra $B^\ell$, where $\ell=1,2$. We have seen that $B^\ell=KQ_{B^\ell}$, where $Q_{B^\ell}$ is the following quiver.
	\begin{enumerate}[(i)]
		\item The set of vertices of $Q_{B^\ell}$ is ${Q_{B^\ell}}_0=Q_0^\ell$.
		\item There is an arrow $i\to j$ in $Q_{B^\ell}$ if and only $\dim \operatorname{Ext}_{A^\ell}(\Delta^\ell(i),\Delta^\ell(j))=1$.
	\end{enumerate}
	Let $C$ be a regular exact Borel subalgebra. Then, $C=KQ_C$ for some quiver $Q_C$ with vertex set $Q_0$. We observe that, by Proposition \ref{proposition:ext between standards in deconcatenation} and \ref{proposition:ext to and from source in deconcatenation}, we have
	$$\dim\operatorname{Ext}_{A^\ell}(\Delta^\ell(i),\Delta^\ell(j))=\dim\operatorname{Ext}_{A}(\Delta(i),\Delta(j)),$$
	for $i,j\in Q_0^\ell$. Using that $B^1$ and $B^2$ are regular exact Borel subalgebras, we get
	$$\dim\operatorname{Ext}^1_{B^\ell}(L(i),L(j))=\dim\operatorname{Ext}^1_{A^\ell}(\Delta^\ell(i),\Delta^\ell(j))=\dim\operatorname{Ext}^1_{A}(\Delta(i),\Delta(j))=\dim\operatorname{Ext}^1_C(L(i),L(j)).$$
	This shows that for any $i,j\in Q_0^\ell$, there is an arrow $i\to j$ in the quiver of $C$ if and only if there is an arrow $i\to j$ in the quiver of $B^\ell$. Next, note that if $i\in Q_0^1\backslash\{v\}$ and $x\in Q_0^2\backslash\{v\}$, then there are no arrows between $i$ and $x$ in $Q_C$, because
	$$\dim \operatorname{Ext}_C^1(L(i),L(x))=\dim \operatorname{Ext}_{A}^1(\Delta(i),\Delta(x))=0,$$
	according to Proposition \ref{proposition:no ext between different parts of concatenation}. This means that the quivers of $C$ and $B^1\diamond B^2$ coincide, so that $C$ and $B^1\diamond B^2$ are isomorphic. This shows that $B^1\diamond B^2$ has $n$ simple modules up to isomorphism and is quasi-hereditary with simple standard modules.
	
	Next, we check that the functor $A\otimes_{B^1\diamond B^2}\blank$ is exact. To this end, we claim that $A$ is projective as a right $B^1\diamond B^2$-module. Consider the decomposition of right $A^1$-modules
	$$A^1\cong \bigoplus_{i\in Q_0^1} P^1(i).$$
	
	Since $A^1\otimes_{B^1}\blank$ is an exact functor, $A^1$ is projective as a right $B^1$-module. Because $B^1\subset A^1$ is a subalgebra, the decomposition above is also a decomposition of right $B^1$-modules. This, in turn, implies that the summands $P^1(i)$ above are projective as right $B^1$-modules. Because $B^1\diamond B^2$ surjects onto $B^1$, each $B^1$-module has the natural structure of a $B^1\diamond B^2$-module (see the remark prior to Lemma \ref{lemma:elem. properties of deconcatenation}). Then, the above decomposition is also a decomposition of right $B^1\diamond B^2$-modules, since the action of $\operatorname{rad}B^2$ on each summand is trivial. Now, consider the decomposition
	$$A\cong \bigoplus_{i\in Q_0^1\backslash\{v\}} P(i) \oplus \bigoplus_{j\in Q_0^2\backslash\{v\}} P(j) \oplus P(v)$$
of right $A$-modules. By the above argument, the summands $P(i)$ in the first term are projective as right $B^1\diamond B^2$-modules (note that for $i\in Q_0^1\backslash\{v\}$, we have $P(i)\cong P^1(i)$). Similarly, the summands of the second term are projective right $B^1\diamond B^2$-modules. Left to check is that $P(v)$ is a projective right $B^1\diamond B^2$-module.

Let $v$ be a sink. Then $P(v)\cong L(v)$ and there is nothing to show. Let $v$ be a source. As noted earlier, $P^\ell(v)$ is projective as a right $B^\ell$-module. Since $v$ is a source also in the quiver of $B^\ell$, the module $P_{B^\ell}^\ell(v)$ is the unique indecomposable projective $B^\ell$-module having $L(v)$ as a composition factor. This implies that there is an isomorphism $f^\ell$ of right $B^\ell$-modules:
$$f^\ell: P^\ell(v)\to P_{B^\ell}^\ell(v) \oplus M^\ell,$$
where $M^\ell$ is a direct sum of indecomposable projective $B^\ell$-modules, with some multiplicites. Note that $M^\ell$ does not contain $P_{B^\ell}^\ell(v)$ as a direct summand. We claim that there is an isomorphism of right $B^1\diamond B^2$-modules
$$f:P(v)\to P_{B^1\diamond B^2}(v) \oplus M^1 \oplus M^2.$$
Define $f$ on paths in $P(v)$ by
$$p\mapsto \begin{cases}
	f^1(p) & \textrm{if }p\in A^1\backslash \operatorname{span}(e_v) \\
	f^2(p) & \textrm{if }p\in A^2\backslash \operatorname{span}(e_v) \\
e_v & \textrm{if }p=e_v.
\end{cases}$$
It is clear that $f$ is bijective and a homomorphism of $B^1\diamond B^2$-modules, hence an isomorphism.

Next, we check that $A\otimes_{B^1\diamond B^2}L(i)\cong \Delta(i)$, for all $i\in Q_0$. When $i\in Q_0^\ell\backslash\{v\}$, we know that $\Delta^\ell(i)\cong \Delta(i)$. In particular, $\Delta(i)$ is an $A^\ell$-module. Let $\varphi:A^\ell\otimes_{B^{\ell}} L(i)\to \Delta(i)$ be an isomorphism. The module $A\otimes_{B^1\diamond B^2}L(i)$ is generated by elements of the form $p\otimes e_i$, where $p$ is a path in $A$. Note that, if $p$ is a path in $A^{\overline{\ell}}$ not equal to $e_v$, then
$$p\otimes e_i=p e_{s(p)} \otimes e_i= p\otimes e_{s(p)}e_i=0.$$
To see this, note that $p$ is a path in $A^{\overline{\ell}}$ and $i\in Q_0^\ell\backslash\{v\}$, so that $i\neq s(p)$. Then, we may define a homomorphism
$$\tilde{\varphi}: A\otimes_{B^1\diamond B^2}L(i)\to \Delta(i)$$ on the non-zero generators of $A\otimes_{B^1\diamond B^2}e_i$ by
$$p\otimes_{B^1\diamond B^2} e_i \mapsto \varphi(p\otimes_{B^\ell} e_i),$$
which is clearly an isomorphism. It remains to check that $A\otimes_{B^1\diamond B^2} L(v)\cong \Delta(v)$. When $v$ is a sink, we have $\Delta(v)\cong L(v)$, so the assertion is clear, given that $A^\ell\otimes_{B^{\ell}} L(v)\cong \Delta(v)$. Assume that $v$ is a source and let
$$f_1: A^1\otimes_{B^1}L(v)\to \Delta^1(v),\quad \textrm{and} \quad f_2:A^2\otimes_{B^2}L(v)\to \Delta^2(v)$$
be isomorphisms. The module $A\otimes_{B^1\diamond B^2} L(v)$ is generated by elements of the form $p\otimes e_v$, where $p\in A$ is some path. Note that with the exception of $e_v$, every such $p$ is contained in $A^1$ or in $A^2$. Define a map
$$f:A\otimes_{B^1\diamond B^2} L(v)\to \Delta(v)$$
on generators by
$$p\otimes e_v \mapsto \begin{cases}
	f_1(p\otimes e_v) & \textrm{if }p\in A^1\backslash \operatorname{span}(e_v),\\
	f_2(p\otimes e_v) & \textrm{if }p\in A^2\backslash \operatorname{span}(e_v),\\
	e_v\otimes e_v & \textrm{if }p=e_v.
\end{cases}$$
	Note that, if $p\in A^1$ and $q\in A^2$, then $p\otimes e_v$ and $q\otimes e_v$ are either zero or linearly independent, except for when $p=q=e_v$. Then, $f$ is a well-defined homomorphism because $f_1$ and $f_2$ are. Similarly, $f$ is injective because $f_1$ and $f_2$ are and the intersection of their images equals $\operatorname{span}(e_v)$. For surjectivity, note that as a vector space, we have $\Delta(v)=\operatorname{span}(e_v)\oplus \operatorname{rad}\Delta^1(v)\oplus \operatorname{rad}\Delta^2(v).$
	
	To finish the proof, note that since $B^1\diamond B^2$ and $C$ are both exact Borel subalgebras, they are conjugate by Theorem~\ref{theorem: uniqueness of borel}. By Theorem~\ref{theorem:inner automorphisms preserve regular exact borel subalgebras}, $B^1\diamond B^2$ is regular since $C$ is.
\end{proof}
We conclude this section with an easy observation about the behavior of (certain) $A_\infty$-structures under deconcatenations.
\begin{proposition}\label{proposition:concatenation of formals is formal}
	Let $Q=Q^1\sqcup Q^2$ be a deconcatenation at a sink or a source $v$ such that if $v$ is a sink, then $v$ is minimal or maximal with respect to $\trianglelefteq^e$, and assume that $\operatorname{Ext}_{A^\ell}(\Delta^\ell,\Delta^\ell)$ is intrinsically formal, for $\ell=1,2$. Then, $\operatorname{Ext}_{A}^\ast(\Delta,\Delta)$ is intrinsically formal.
\end{proposition}
\begin{proof}
	Let $m_n$ denote the higher multiplications on $\operatorname{Ext}_A^\ast(\Delta,\Delta)$ and consider $m_n(\varphi_n,\dots, \varphi_1)$, where $\varphi_i\in \operatorname{Ext}_A^\ast(\Delta,\Delta)$ for $1\leq i \leq n$. Suppose that $m_n(\varphi_n,\dots, \varphi_1)\in \operatorname{Ext}_A^k(\Delta(a),\Delta(b))$, for some $a, b\in Q_0$ and $k\geq 0$. If $a,b\in Q_0^\ell$, then $m_n(\varphi_n,\dots, \varphi_1)=0$, because $\operatorname{Ext}_{A^\ell}^\ast(\Delta^\ell,\Delta^\ell)$ is intrinsically formal. If $a\in Q_0^{\ell}\backslash\{v\}$ and $b\in Q_0^{\ell}\backslash\{v\}$, then $m_n(\varphi_n,\dots, \varphi_1)=0$, according to Proposition~\ref{proposition:no ext between different parts of concatenation}. If $a=b=v$, then $m_n(\varphi_n,\dots,\varphi_1)=0$, since standard modules have no non-split self-extensions.
\end{proof}
\subsection{Regular exact Borel subalgebras for path algebras of linear quivers}
Let $Q$ be the linear quiver
$$\xymatrix{1\ar@{-}[r] & 2 \ar@{-}[r] &\dots \ar@{-}[r] & n-1 \ar@{-}[r] & n}$$
where the edges may be of either orientation, and put $\Lambda=KQ$. Denote by $\mathbf{A}_m$ and $\mathbf{B}_\ell$ the special cases
$$\xymatrix{1 \ar[r] & \dots \ar[r] & m}\quad\textrm{and}\quad \xymatrix{1 & \dots \ar[l] & \ell\ar[l]},$$
respectively. By Proposition \ref{proposition:A_n has a regular exact borel}, $K\mathbf{A}_m$ admits a regular exact Borel subalgebra, regardless of the partial order on the vertices. We remark that all the statements from Section 2, describing the structure of $K\mathbf{A}_m$ in terms of its binary search tree, remain true for $K\mathbf{B}_\ell$, if we swap ``left'' and ``right''. For instance, over $K\mathbf{B}_\ell$, the composition factors of a standard module $\Delta_{K\mathbf{B}_\ell}(i)$ are found in the \emph{left} subtree of the vertex labeled by $i$, rather than the right. From this, it follows that also $K\mathbf{B}_\ell$ admits a regular exact Borel subalgebra, regardless of the partial order on the vertices. Now, consider $\Lambda=KQ$. The quiver $Q$ admits an iterated deconcatenation
$$Q=Q^1\sqcup \dots \sqcup Q^s,$$
where, for each $1\leq r\leq s$, the quiver $Q^r$ is either $\mathbf{A}_m$ or $\mathbf{B}_\ell$. Recall that the partial order on the set of vertices of $Q$ is constructed in the following way. When $Q=Q^1\sqcup Q^2$, with $Q_0^1\cap Q_0^2=\{v\}$ and $i,j\in Q_0$, we say that $i\triangleleft j$ if one of the following hold.
\begin{enumerate}[(i)]
	\item We have $i,j\in Q_0^\ell$ and $i\triangleleft^\ell j$, for some $\ell$.
	\item We have $i\in Q_0^\ell$, $j\in Q_0^{\overline{\ell}}$, $i\triangleleft^\ell v$ and $v\triangleleft^{\overline{\ell}} j$.
\end{enumerate}
It is clear that, in this setting, $v$ is maximal in $Q_0$ if and only if it is maximal in both $Q_0^\ell$ and $Q_0^{\overline{\ell}}$. We conclude that we have a special case of Proposition \ref{proposition:gluing of deconcatenation admits regular exact borel} as follows.

\begin{theorem}\label{theorem:when does path algebra of linear quiver have borel}
	The algebra $\Lambda$ admits a regular exact Borel subalgebra $C\subset \Lambda$ if and only if each vertex $v\in Q_0$ which is a sink at which a deconcatenation occurs is minimal or maximal with respect to the essential order on $Q_0$.
\end{theorem}
\begin{example}
	Consider the quiver
	$$Q=\xymatrix{
	1 \ar[r] & 2 \ar[r] & 3 & 4 \ar[l] & 5\ar[l] \ar[r] & 6 \ar[r]	&7.
	}$$
Deconcatenating $Q$ as far as possible we find that
$$Q=(\xymatrix{1\ar[r]^-x & 2 \ar[r]^-y & 3})\sqcup(\xymatrix{3 & 4 \ar[l]_-a & 5\ar[l]_-b})\sqcup (\xymatrix{5 \ar[r]^-u & 6 \ar[r]^-v & 7}).$$

Suppose the order on $Q_0$ is the one given by the following orders on $Q_0^1, Q_0^2$ and $Q_0^3$.

$$1\triangleleft_T 3, 2\triangleleft_T 3, \quad 4\triangleleft_T 5 \triangleleft_T3,\quad 6\triangleleft_T 5 \triangleleft_T 7,$$
corresponding to binary search trees
	$$
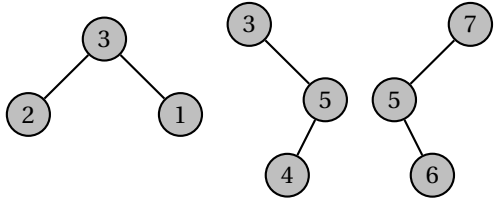
\begin{tikzpicture}
	\node(a) [shape=circle, draw, thick, fill=lightgray] at (0,0) {3};
	\node(b) [shape=circle, draw, thick, fill=lightgray] at (-1,-1) {2};
	\node(c) [shape=circle, draw, thick, fill=lightgray] at (1,-1) {1};
	\node(d) at (0,-2) {};
	\draw[thick] (a) to (b);
	\draw[thick] (a) to (c);
\end{tikzpicture}\quad \begin{tikzpicture}
\node(a) [shape=circle, draw, thick, fill=lightgray] at (0,0) {3};
\node(b) [shape=circle, draw, thick, fill=lightgray] at (1,-1) {5};
\node(c) [shape=circle, draw, thick, fill=lightgray] at (0.5, -2) {4};
\draw[thick] (a) to (b);
\draw[thick] (b) to (c);
\end{tikzpicture}\quad  \begin{tikzpicture}
\node(a) [shape=circle, draw, thick, fill=lightgray] at (0,0) {7};
\node(b) [shape=circle, draw, thick, fill=lightgray] at (-1, -1) {5};
\node(c) [shape=circle, draw, thick, fill=lightgray] at (-0.5, -2) {6};
\draw[thick] (a) to (b);
\draw[thick] (b) to (c);
\end{tikzpicture}$$
We compute the standard modules.
$$\Delta(1)\cong L(1), \quad \Delta(2)\cong L(2),\quad \Delta(3)\cong L(3), \quad \Delta(4)\cong L(4),\quad \Delta(5)\cong M(4,5),\quad \Delta(6)\cong L(6),\quad  \Delta(7)\cong L(7)$$
The only standard module with non-trivial radical is $\Delta(5)$ with $\operatorname{rad}\Delta(5)\cong L(4)\cong \nabla(4)$, so $\operatorname{rad}\Delta(i)\in \mathcal{F}(\nabla)$ for all $1\leq i\leq 7$.
Then, $A^1, A^2$ and $A^3$ have regular exact Borel subalgebras given by

$$(\xymatrix{1\ar[r]^-x & 2 & 3}),\quad (\xymatrix{ 3 & 4\ar[r]^-b & 5}),\quad (\xymatrix{5\ar@/^1pc/[rr]^-{vu} & 6 & 7}).$$
We note that, since $3$ is maximal in the first and second quivers, $A$ has a regular exact Borel subalgebra, by Proposition \ref{proposition:gluing of borels is borel}, given by the following quiver.
$$\xymatrix{
1 \ar[r]^-x & 2 & 3 & 4\ar[r]^-b & 5 \ar@/^1pc/[rr]^-{vu} & 6 & 7	
}.$$
\end{example}
Similarly to the above, we now wish to extend the statement of Proposition \ref{proposition:a-infty on A_n is trivial} to the algebra $\Lambda$ discussed in this section.
\begin{proposition}\label{proposition:a-infty on linear quiver is trivial}
$\operatorname{Ext}_{\Lambda}^\ast(\Delta,\Delta)$ is intrinsically formal.
\end{proposition}
\begin{proof}
This follows immediately from Proposition~\ref{proposition:a-infty on A_n is trivial} and Proposition~\ref{proposition:concatenation of formals is formal}.
\end{proof}
In view of the situations considered in \cite{FKR}, it would be interesting to do as similar investigation with quivers of Dynkin type $\mathbb{D}$ and $\mathbb{E}$.

\subsection{A nonlinear example}
Consider the algebra $A$, given by the quiver

$$
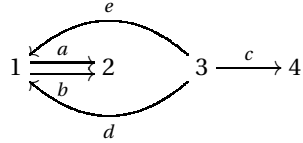
\xymatrix{
	1 \ar@<0.5ex>[r]^-a \ar@<-0.5ex>[r]_-{b} & 2 & 3\ar[r]^-c \ar@/^1.5pc/[ll]^-d \ar@/_1.5pc/[ll]_-e& 4}$$
subject to the relations $ae=bd$, $ad=0$ and $be=0$. With the usual order $1\triangleleft2\triangleleft3\triangleleft4$, the standard modules over $A$ are
$$\Delta(1)\cong L(1),\quad \Delta(2)\cong L(2),\quad \Delta(3):\vcenter{ \xymatrixrowsep{0.5cm}\xymatrixcolsep{0.5cm}
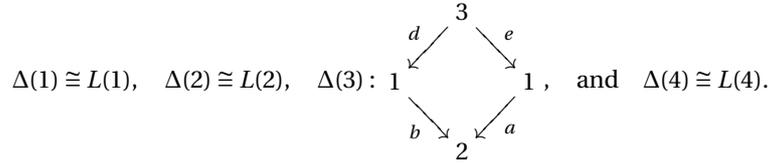
\xymatrix{
& 3 \ar[ld]_-d \ar[rd]^-e \\
1 \ar[rd]_-b & &1\ar[ld]^-a\\
& 2	
}},\quad \textrm{and} \quad \Delta(4)\cong L(4).$$
The only standard module with non-trivial radical is $\Delta(3)$, whose radical is isomorphic to the costandard module $\nabla(2)$. Therefore, $A$ admits a regular exact Borel subalgebra $B\subset A$, which we find is given by the quiver
$$
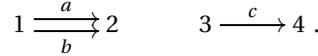
\xymatrix{
	1 \ar@<0.5ex>[r]^-a \ar@<-0.5ex>[r]_-{b} & 2 & 3 \ar[r]^-c& 4}.$$
Consider now the algebra $\Lambda$ given by the quiver
$$
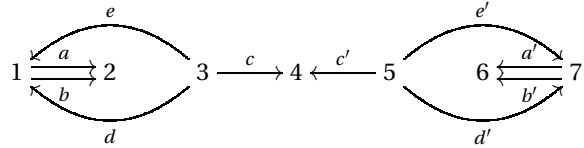
\xymatrix{
	1 \ar@<0.5ex>[r]^-a \ar@<-0.5ex>[r]_-{b} & 2 & 3\ar[r]^-c \ar@/^1.5pc/[ll]^-d \ar@/_1.5pc/[ll]_-e& 4 & \ar[l]_-{c^\prime}5 \ar@/^1.5pc/[rr]^-{e^\prime} \ar@/_1.5pc/[rr]_-{d^\prime} & 6 & 7 \ar@<-0.5ex>[l]_-{a^\prime} \ar@<0.5ex>[l]^-{b^\prime}}$$
subject to the relations
$$ae=bd, \quad a^\prime e^\prime=b^\prime d^\prime,\quad ad=0,\quad a^\prime d^\prime=0,\quad be=0,\quad \textrm{and}\quad b^\prime e^\prime=0,$$
with the order on the vertices being $1<2<3<4, \quad 7<6<5<4$. Then, $4$ is maximal and Proposition \ref{proposition:sufficient and necessary condition for borel when v is a sink and costandard at v not simple} applies, and $\Lambda$ has a regular exact Borel subalgebra $C\subset \Lambda$, given by the quiver
$$
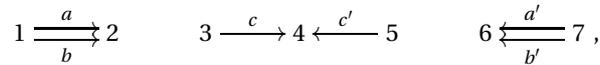
\xymatrix{1 \ar@<0.5ex>[r]^-a \ar@<-0.5ex>[r]_-{b}& 2 & 3\ar[r]^-{c} & 4 & 5\ar[l]_-{c^\prime} & 6 & 7 \ar@<-0.5ex>[l]_-{a^\prime} \ar@<0.5ex>[l]^-{b^\prime}},$$
according to Proposition \ref{proposition:gluing of borels is borel}.

\section{Acknowledgements}
\subsection*{Funding} Open access funding provided by Uppsala University. The author did not receive support from any organization for the submitted work.
\subsection*{Competing interests} The author has no competing interesets to declare that are relevant to the content of this article.
\newpage\printbibliography
\end{document}